\documentclass{article}

\usepackage{amsfonts, amsmath, amsthm, bbm, amssymb, color, enumerate, physics,verbatim}
\usepackage{ulem}
\usepackage{dsfont}
\newtheorem{defi}{Definition}[section]
\newtheorem{lem}{Lemma}[section]

\newtheorem{prop}{Proposition}[section]
\newtheorem{thm}{Theorem}[section]

\def\f{\frac}
\def\d{\mathrm d}
\def\e{\mathrm e}
\def\i{\mathrm i}

\def\R{\mathbb{R}}
\def\N{\mathbb{N}}
\def\C{\mathbb{C}}
\def\D{\mathbb{D}}
\def\A{\mathcal{A}}

\def\cL{\mathcal{L}}
\def\cO{\mathcal{O}}
\def\cQ{\mathcal{Q}}
\def\cW{\mathcal{W}}
\def\cD{\mathcal{D}}
\def\cB{\mathcal{B}}
\def\cC{\mathcal{C}}

\def\cS{\mathcal{S}}

\newcommand{\bel}{\begin{equation} \label}
\newcommand{\ee}{\end{equation}}
\def\beq{\begin{equation}}
\def\eeq{\end{equation}}
\newcommand{\bea}{\begin{eqnarray}}
\newcommand{\eea}{\end{eqnarray}}
\newcommand{\beas}{\begin{eqnarray*}}
\newcommand{\eeas}{\end{eqnarray*}}

\def\de{\delta}
\def\De{\Delta}
\def\al{\alpha}
\def\be{\beta}
\def\ga{\gamma}
\def\Ga{\Gamma}
\def\ep{\epsilon}
\def\ve{\varepsilon}
\def\ka{\kappa}
\def\te{\theta}
\def\la{\lambda}

\def\vp{\varphi}
\def\na{\nabla}
\def\om{\omega}
\def\Om{\Omega}

\def\Sg{\Sigma}

\def\ov{\overline}
\def\pa{\partial}
\def\Re{\mathrm{\,Re\,}}
\def\Im{\mathrm{\,Im\,}}
\def\wh{\widehat}
\def\wt{\widetilde}

\def\LL{{L^2(\Om)}}
\def\HH{{H^2(\Om)}}

\def\dt{\d t}
\def\ds{\d s}
\def\dr{\d r}
\def\da{\d \al}

\title{\bf
Initial-boundary value problem for distributed order time-fractional diffusion equations
}
\author{
Zhiyuan LI$^*$,
Yavar KIAN$^\dag$ and
Eric SOCCORSI$^\dag$
}
\date{}

\begin{document}
\maketitle
\renewcommand{\thefootnote}{\fnsymbol{footnote}}
\footnotetext{\hspace*{-5mm} 
\begin{tabular}{@{}r@{}p{13cm}@{}} 
& Manuscript last updated: \today.
\\
$^\dag$& Aix-Marseille Univ, Universit\'e de Toulon, CNRS, CPT, Marseille, France.
\\
$^*$& Graduate School of Mathematical Sciences, 
the University of Tokyo, 
3-8-1 Komaba, Meguro-ku, Tokyo 153-8914, Japan. 
E-mail:
zyli@ms.u-tokyo.ac.jp.
\end{tabular}}

\begin{abstract}
We examine initial-boundary value problems for diffusion equations with distributed order time-fractional derivatives. We prove existence and uniqueness results for the weak solution to these systems, together with its continuous dependency on initial value and source term. Moreover, under suitable assumption on the source term, we establish that the solution is analytic in time.
\end{abstract}

\section{Introduction}
\label{sec-intro-distri}

The time-fractional diffusion model of constant order (CO) $\al$, $\pa_t^\al u - \De u = f$, has received great attention within the last few decades from numerous applied scientists, see e.g. \cite{BWM00,CSL13,HH98,LB03}, due to its relevance for modeling anomalous diffusion processes whose mean square displacement (MSD) scales like $t^\al$ as the time variable $t$ goes to infinity. But, more recently, it was noticed in \cite{CRSG03,N04,SCK04} and the references therein that several application areas such as polymer physics or kinetics of particles moving in quenched random force fields, exhibit ultraslow diffusion phenomena whose MSD is of logarithmic growth. There are several approaches for modeling such processes. One of them uses time-fractional diffusion equations of distributed order (DO), see e.g. \cite{K08, MMPG08, RZ16}, but we also mention that a diffusion model with variable fractional order time-derivative was proposed in \cite{SCC09} to depict ultraslow diffusion processes. In this paper, we are concerned with the DO fractional diffusion model. More precisely, we consider the following initial-boundary value problem (IBVP),
\begin{equation}
\label{equ-u-distri}
\left\{
\begin{alignedat}{2}
&\D^{(\mu)}_t u +\A u = F
&\quad& \mbox{in}\ Q := \Om \times (0,T),\\
&u(0,\cdot)=u_0 &\quad& \mbox{in $\Om$},
\\
&u=0
&\quad& \mbox{on}\ \Sg :=\pa\Om\times(0,T),
\end{alignedat}
\right.
\end{equation}
where $\Om$ is an open bounded domain in $\R^d$, $d \ge 2$, with $\cC^{1,1}$ boundary $\pa\Om$, and $\A$ is the following symmetric and uniformly elliptic operator
$$
\A\vp(x)
:=-\sum_{i,j=1}^d \pa_{x_i} (a_{ij}(x)\pa_{x_j} \vp(x))+q(x)\vp(x),
$$
associated with a suitable real-valued electric potential $q$ and symmetric coefficients
$a_{ij}= a_{ji}$, $1\leq i,j \leq d$, fulfilling the uniform ellipticity condition
\begin{equation}
\label{a-ell}
\sum_{i,j=1}^d a_{ij}(x)\xi_i \xi_j \ge c_a \abs{\xi}^2 ,\
x \in\ov\Om,\ \xi=(\xi_1,\cdots,\xi_d ) \in \R^d,
\end{equation}
for some positive constant $c_a$.
Here, $\D^{(\mu)}_t$ denotes the distributed order fractional derivative
$$
\D^{(\mu)}_t \vp(t) := \int_0^1 \pa^\al_t \vp(t) \mu(\al) \da,
$$
induced by the non-negative
weight function $\mu \in L^\infty(0,1)$, and $\pa_t^\al$ is the Caputo derivative of order $\al$, defined by:
$$
\pa^\al_t \vp(t) := 
\left\{
\begin{alignedat}{2}
&\vp(t) &\quad& \mbox{if}\ \al = 0,
\\
&\f{1}{\Ga(1-\al)}\int_0^t \f{\vp'(\tau)}{(t-\tau)^\al}\d\tau
&\quad& \mbox{if}\ \al \in (0,1),
\\
&\vp'(t)&\quad& \mbox{if}\ \al=1,
\end{alignedat}
\right.
$$
the sysmbol $\Gamma$ denoting the usual Gamma function.

From a mathematical viewpoint, the forward problem associated with these equations was investigated in \cite{K08, LLY15, L-FCAA, MNV11}.
Namely, the fundamental solution to the 
Cauchy problem for both ordinary and partial distributed order fractional differential equations with continuous weight function was derived and investigated in detail in \cite{K08}. A uniqueness result for the solution to diffusion equations of DO was derived in \cite{L-FCAA} with the aid of an appropriate maximum principle and a formal solution was constructed by means of the Fourier method of variables separation. Unfortunately, there is no proof available in \cite{L-FCAA} of the convergence of the series describing this formal solution. Further, explicit strong solutions (and stochastic analogues) to DO time-fractional diffusion equations with Dirichlet boundary conditions were built in \cite{MNV11} for $C^1$-weight functions.
For the asymptotic behavior we refer to \cite{LLY15}, where logarithmic decay of the solution to DO diffusion equations was established for $t\to\infty$, while this solution scales at best like $(t | \log t | )^{-1}$ as $t \to 0$. These results are in sharp contrast with the ones derived for single- and multi-term time-fractional diffusion equations, see e.g.  \cite{KSVZ, LLY-AMC, SY11}, which seem unable to capture the time asymptotic trends of ultraslow diffusion processes.

Formally, single- or multi-term time-fractional diffusion equations can be seen as DO time-fractional diffusion equations associated with a density function of the form $\mu=\sum_{j=1}^\ell q_j\de(\cdot-\al_j)$, where $\de$ is the Dirac-delta function. The definition of a weak solution for single- or multi-term time-fractional diffusion equations was recently introduced in \cite{GLY15,JLLY} by taking advantage of the fact that the domain of the $L^2(0,T)$-realization of the operator $\pa_t^\al$, $\al \in (0,1)$, is embedded in the fractional Sobolev space of order $\al$, $H^\al(0,T)$. However, as it is still unclear whether the domain of $\D^{(\mu)}_t$ can be described by fractional Sobolev spaces, the scheme developed in \cite{GLY15,JLLY} does not seem to be relevant for DO time-fractional diffusion equations. Another approach, initiated by \cite{Z09} and recently applied to DO time-fractional diffusion equations in \cite{KR}, is to consider \eqref{equ-u-distri} as an abstract evolutionary integro-differential equation. This strategy is suitable for both autonomous and non-autonomous equations but there is a serious inconvenience to this method, arising from the dependency of the functional space of the solution on the kernel function of the corresponding integro-differential operator, which causes numerous technical difficulties in performing computations based on this model.

In this paper, since the system under study is autonomous (the coefficients appearing in \eqref{equ-u-distri} are all space-dependent only), we rather follow the idea of \cite{KY15} and characterize the weak solution to \eqref{equ-u-distri} as the original of the solution to the Laplace transform of \eqref{equ-u-distri} with respect to the time variable. With reference to the analysis carried out in \cite{LLY-AMC} for multi term CO time-fractional diffusion equations, we aim to study the existence, uniqueness and regularity properties, and the stability with respect to the diffusion coefficients and the weight function $\mu$, of a weak solution to \eqref{equ-u-distri}.

\subsection{Settings}

In this paper, we assume that $q \in L^\kappa(\Om)$ for some $\kappa \in (d,+\infty)$, is non-negative, i.e.
\begin{equation}
\label{p-nn}
q(x) \geq 0,\ x \in \Om,
\end{equation}
and that the coefficients
$a_{i,j}= a_{j,i}\in \cC^1(\ov\Om,\R)$, $1\leq i,j \leq d$, satisfy the ellipticity condition \eqref{a-ell}.

We denote by $A$ the operator generated in $\LL$ by the quadratic form
$$ 
u \mapsto \sum_{i,j=1}^d \int_{\Om} a_{i,j}(x) \pa_{x_i} u(x) \pa_{x_j} \overline{u(x)} \d x + \int_{\Om} q(x) \abs{u(x)}^2 \d x ,\ u \in H_0^1(\Om). 
$$
Due to \eqref{a-ell}, $A$ is selfadjoint in $\LL$ and acts as the operator $\A$ on its domain $D(A):=H_0^1(\Om) \cap \HH$,
in virtue of \cite[Section 2.1]{KSY}.

Moreover, since $H_0^1(\Om)$ is compactly embedded in $\LL$, the resolvent of $A$ is compact in $\LL$, hence the spectrum of $A$ is purely discrete. We denote by $\{\la_n,\ n \in \N \}$, where $\N := \{ 1, 2, \ldots \}$, the sequence of the eigenvalues of $A$ arranged in non-decreasing order and repeated according to their multiplicity.
In light of \eqref{a-ell}-\eqref{p-nn}, we have
\begin{equation}
\label{non-degeneracy}
\la_n \ge c_a>0,\ n \in \N.
\end{equation}
For further use, we introduce an orthonormal basis $\{ \vp_n,\ n \in \N \}$ of eigenfunctions of $A$ in $\LL$, such that
$$ 
A \vp_n=\la_n\vp_n,\ \vp_n \in D(A),\ n \in \N. 
$$

\subsection{Weak solution}

As already mentioned in the introduction, the usual definition given in \cite{GLY15} of a weak solution to CO time-fractional diffusion equations, is not suitable for DO time-fractional diffusion equations. Hence we rather follow the strategy implemented in \cite{KSY} (which is by means of the Laplace transform of tempered distributions), that is recalled below. 
 
Let $\cS'(\R,\LL) := \cB(\cS(\R,\LL),\R)$ be the space dual to $\cS(\R;\LL)$, put $\R_+:= [0,+\infty)$, and denote by
$\cS'(\R_+,\LL):=\{v \in \cS'(\R,\LL);\ \mbox{supp}\ v \subset \R_+ \times \ov{\Om}\}$ the set of distributions in $\cS'(\R,\LL)$ that are supported in $\R_+ \times \ov{\Om}$. Otherwise stated, $v \in \cS'(\R,\LL)$ lies in $\cS'(\R_+,\LL)$ if and only if
$\langle v , \vp \rangle_{\cS'(\R,\LL),\cS(\R,\LL)} =0$
whenever $\vp \in \cS(\R,\LL)$ 
vanishes in $\R_+ \times \ov{\Om}$. As a consequence we have
$$
\langle v(\cdot,x) , \vp \rangle_{\cS'(\R),\cS(\R)}
=\langle v(\cdot,x)  , \psi \rangle_{\cS'(\R),\cS(\R)},\ \vp, \psi \in \cS(\R),
$$
provided $\vp=\psi$ in $\R_+ $. 
Further, we say that $\vp \in \cS(\R_+)$ if $\vp$ is the restriction to $\R_+$ of a function $\wt{\vp}\in \cS(\R)$. Then, for a.e. $x\in\Om$, we set
$$
\langle v(\cdot,x) , \vp \rangle_{\cS'(\R_+),\cS(\R_+)} := x \mapsto \langle v(\cdot,x)  , \wt{\vp} \rangle_{\cS'(\R),\cS(\R)},\ v \in \cS'(\R_+,\LL).
$$
Notice that $\wt{\vp}$ may be any function in $\cS(\R)$ such that $\wt{\vp}(t)=\vp(t)$ for all $t \in \R_+$.

For all $p \in \C_+ := \{ z\in\C;\ \Re z \in (0,+\infty) \}$, where $\Re z $ denotes the real part of $z$, we put
$$ 
e_p(t) := \exp (-p t ),\ t \in \R_+.
$$
Evidently, $e_p$ lies in $\cS(\R_+)$ so we can define the Laplace transform $\cL[v]$ of $v \in \cS'(\R_+,\LL)$, with respect to $t$, as the family of mappings
$$
\cL [v](p) : x \in \Om \mapsto \langle v(\cdot,x) , e_p \rangle_{\cS'(\R_+),\cS(\R_+)},\ p \in \C_+.
$$ 
Notice that $p \mapsto \cL [v](p) \in \cC^\infty(\C_+,\LL)$. 

Inspired by \cite[Definition 2.2]{KSY} we may now introduce the weak solution to \eqref{equ-u-distri} as follows.

\begin{defi}
\label{def-u-distri}
Let $\mu \in L^\infty(0,1;\R_+)$, let $a\in\LL$ and, depending on whether $T \in (0,+\infty)$ or $T=+\infty$, let $F \in L^\infty(0,T;\LL)$ or let $t \mapsto (1+t)^{-m}F(t,\cdot) \in L^\infty(\R_+,\LL)$ for some $m\in \N_0 := \N \cup \{ 0 \}$.  We say that $u$ is a {\it weak solution} to the IBVP \eqref{equ-u-distri} if $u$ is the restriction to $Q$ of a distribution $v\in \cS'(\R_+,\LL)$, i.e. $u=v|_{Q}$, whose Laplace transform $\wh v:=\cL[v]$ fulfills
\begin{equation}
\label{eq1}
(A+sw(s)) \wh v(s) = w(s) u_0 + \wh F(s),\ s \in \C_+,
\end{equation}
with $w(s):=\int_0^1 s^{\al-1} \mu(\al)\da$ and
$\wh F(s):=\cL[F\mathbbm{1}_{(0,T)}](s)=\int_0^T\e^{-st}F(t,\cdot)\dt$.
Here $\mathbbm{1}_{(0,T)}$ stands for the characteristic function of $(0,T)$.
\end{defi}
We stress out that Equation \eqref{eq1} imposes that $\wh v(s) \in D(A)$ for all $s \in \C_+$.

In the coming section we state several existence and uniqueness results for the weak solution to \eqref{equ-u-distri}.

\subsection{Main results}

We first address the case where the initial state $u_0 \in \LL$. The corresponding result is as follows. 
 
\begin{thm}
\label{thm-forward} 
Let $\mu \in L^\infty(0,1)$ be non-negative and fulfill the following condition: 
\begin{equation}
\label{cnd-mu1}
\exists \al_0 \in(0,1),\ \exists \de \in (0,\al_0),\ \forall \al \in (\al_0-\de,\al_0),\ \mu(\al) \ge \f{\mu(\al_0)}{2}>0.
\end{equation}
Depending on whether $T \in (0,+\infty)$ or $T=+\infty$, we assume either that $F\in L^\infty(0,T;\LL)$ or that $t \mapsto (1+t)^{-m}F(t,\cdot) \in L^\infty(\R_+,\LL)$ for some natural number $m \in \N_0$. 

Then, for all $u_0 \in \LL$ and all $F \in L^\infty(0,T;\LL)$, there exists a unique weak solution $u \in \cC((0,T],\LL)$ to \eqref{equ-u-distri}. Moreover, we have $u \in \cC([0,T],\LL)$ provided $u_0=0$. 
\end{thm}

Here and in the remaining part of this text, the notation $(0,T]$ (resp., $[0,T]$) stands for $(0,+\infty)$ (resp., $[0,+\infty)$) in the particular case where $T=+\infty$. 
For a finite final time $T$, we have the following improved regularity result.

\begin{thm}
\label{thm-L1}
Fix $T \in (0,+\infty)$ and let $\mu$ be the same as in Theorem \ref{thm-forward}. 
\begin{enumerate}[(a)]
\item Pick $u_0 \in D(A^\ga)$ for $\ga \in (0,1]$, and let $F=0$. Then, the unique weak solution $u$ to \eqref{equ-u-distri} lies in 
$\cC([0,T],\LL)\cap \cC((0,T],H_0^1(\Om) \cap \HH)$ and satisfies the two following estimates
\begin{equation}
\label{esti-u}
\|u(t)\|_\HH
\leq C \e^{T} \|u_0\|_{D(A^\ga)} t^{\ga-1},\ t \in (0,T],
\end{equation}
and 
\begin{equation}
\label{esti-u_t}
\|\pa_tu(t)\|_\LL
\leq C \e^{T}\|u_0\|_{D(A^\ga)} t^{-\be},\ \be \in \left( 1- \al_0 \ga, 1 \right),\ t \in (0,T],
\end{equation}
where $C$ is a positive constant which is independent of $T$, $t$ and $u_0$.
\item Assume that $u_0=0$ and $F \in L^\infty(0,T;\LL)$. Then, for all $\ka\in[0,1)$ and all $p \in  \left[ 1,\f1{1-\al_0(1-\ka)} \right)$, the unique weak solution $u$ to \eqref{equ-u-distri} lies in $L^p(0,T;H^{2\ka}(\Om))$. Moreover, there exists a positive constant $C$, independent of $T$ and $F$, such that we have
\begin{equation}
\label{esti_u-h2}
\|u\|_{L^p(0,T;H^{2\ka}(\Om))} \le C \e^{C T} \|F\|_{L^1(0,T;\LL)}.
\end{equation} 
\end{enumerate}
\end{thm}

Notice that the second claim of Theorem \ref{thm-forward} states that $u \in L^p(0,T; D(A^\ka))$ for suitable values of $p$, provided $u_0=0$ and $F\in L^\infty(0,T;\LL)$. Here $\ka$ can be arbitrarily close to $1$ without actually becoming $1$, 
except for the special case treated by the following result, where the density function 
$\mu$ vanishes in a neighborhood of the endpoint of the interval $(0,1)$. 


\begin{thm}
\label{thm-H2}
Let $\mu \in \cC([0,1],\R_+)$ satisfy \eqref{cnd-mu1}. Assume moreover that there exists $\alpha_1 \in (\alpha_0,1)$, where $\alpha_0$ is the same as in \eqref{cnd-mu1}, such that
\begin{equation}
\label{cnd-mu2}
\mu(\alpha)=0,\ \alpha \in (\alpha_1,1).
\end{equation}
Then, for $u_0=0$ and $F\in L^\infty(0,T;\LL)$, there exists a unique solution $u \in L^2(0,T;\HH)$ to the IBVP \eqref{equ-u-distri}, satisfying
$$
\norm{u}_{L^2(0,T;\HH)} \leq C \norm{F}_{L^2(Q)},
$$
for some positive constant $C$, which is independent of $F$.
\end{thm}

It turns out that the statement of Theorem \ref{thm-H2} remains valid without the technical assumption \eqref{cnd-mu2}, but this is at the price of greater difficulties in the derivation of the corresponding result. Hence, in order to avoid the inadequate expense of the size of this article, we shall go no further into this direction.

The following result claims that the solution to \eqref{equ-u-distri} is analytic in time, provided the source term can be holomorphically extended to a neighborhood of the positive real axis. This statement is of great interest in the analysis of inverse coefficient problems associated with time-fractional diffusion equations, see e.g. \cite{KOSY,KSY,LIY}.
\begin{thm}
\label{thm-an} 
Let the conditions of Theorem \ref{thm-forward} be satisfied with $T=+\infty$. Assume moreover that there exists $\rho \in (0,+\infty)$ such that the source term $t \mapsto F(t,\cdot)$ is extendable to a holomorphic function of the half-strip $S_\rho := \{ z \in \C; \Re z \ \in(-\rho,+\infty),\ \Im z \in (-\rho,\rho) \}$ into $\LL$, where we recall that $\Re z$ (resp., $\Im s$) stands for the real (resp., imaginary) part of $z$. Then, the weak-solution $t \mapsto u(t,\cdot)$ to \eqref{equ-u-distri}, given by Theorem \ref{thm-forward}, can be extended to an analytic function of $(0,+\infty)$ into $\LL$. 
\end{thm}

The last result of this paper is similar to \cite[Theorem 2.3]{LLY-AMC}, which was established in the framework of multi-terms time-fractional diffusion equations with positive constant coefficients. It is useful for the optimization approach to the inverse problem of determining the weight function $\mu$ together with the diffusion matrix $a:=(a_{i,j})_{1 \le i,j \le d}$ and the electric potential $q$, appearing in the definition of the operator $\A$, by extra data of the solution to  \eqref{equ-u-distri}. Namely, we claim for all {\it a priori} fixed $M \in (0,+\infty)$ that the weak solution to the IBVP \eqref{equ-u-distri} associated with $F=0$, depends Lipschitz continuously on $(\mu,a,q)$ in $\cW \times \cD(M) \times \cQ(M)$, where 
$\cW :=\left\{ \mu\in L^\infty(0,1;\R_+)\ \mbox{satistying}\ \eqref{cnd-mu1} \right\}$
is the set of admissible weight functions, 
$$ \cD(M)  := \{ a = a^T = (a_{i,j})_{1 \le i,j \leq d} \in C^1(\ov\Om,\R^{d^2})\ \mbox{fulfilling}\ \eqref{a-ell}\ \mbox{and}\ \norm{a}_{C^1(\ov\Om)}\le M \}, $$
with $a^T$ the transpose matrix of $a$, denotes the set of admissible diffusion matrices and 
$$ \cQ(M) := \{ q \in C^0(\ov\Om)\ \mbox{obeying}\ \eqref{p-nn}\ \mbox{and}\ \norm{q}_{C^0(\ov\Om)}\le M \}, $$
is the set of admissible electric potentials.

\begin{thm}
\label{thm-Lip-coef}
Let $T \in (0,+\infty)$, fix $M \in (0,+\infty)$ and pick $(\mu,a,q)$ and $(\wt \mu, \wt a,\wt q)$ in $\cW \times \cD(M) \times \cQ(M)$. For $u_0 \in D(A^\ga)$, where $\ga \in (0,1]$ is fixed, let
$u$ (resp., $\wt u$) denote the weak solution to the IBVP
\eqref{equ-u-distri} (resp., the IBVP \eqref{equ-u-distri} where $(\wt \mu,\wt a, \wt q)$ is substituted for $(\mu,a,q)$) with uniformly zero source term, given by Statement (a) in Theorem \ref{thm-L1}.

Then, for all $\ka\in(0,1)$ and all $p \in \left[1, \f{1}{1-\al_0(1-\ka)} \right)$, there exists a constant $C \in (0,+\infty)$, depending only on $T$, $M$, $\mu$, $\ga$, $\ka$, $p$ and $c_a$, such that we have
\begin{equation}
\label{eq-Lip-coef}
\|u-\wt u\|_{L^p(0,T;H^{2 \ka}(\Om))}
\le C \left( \|\mu-\wt\mu\|_{L^\infty(0,1)}
+\|a-\wt a\|_{C^1(\ov\Om)} + \|q-\wt q\|_{C^0(\ov\Om)}\right).
\end{equation}
\end{thm}

\subsection{Generalization of the results}

In Theorems \ref{thm-forward}, \ref{thm-L1} and \ref{thm-H2}, the function $F$ is assumed to be bounded in time over $(0,T)$ but it turns out that
in many applications the relevant source term appearing in \eqref{equ-u-distri} lies in $L^1(0,T;\LL)$.
In such a case the Definition \ref{def-u-distri} of a weak solution to the IBVP \eqref{equ-u-distri} is no longer valid but it can be easily adapted to the framework of $F \in L^1(0,T;\LL)$ with a density argument we make precise below.

Let us start by introducing the following notation: for all $\varphi \in L^\infty(0,T;\LL)$ we denote by $u(\varphi)$ the $L^1(0,T;H_0^1(\Om))$-solution to the system \eqref{equ-u-distri} associated with $u_0=0$ and $F=\varphi$, which is given by Statement (b) in Theorem \ref{thm-L1} with $p=1$ and $\kappa=\f{1}{2}$.
Then, for $T \in (0,+\infty)$, $u_0=0$ and $F \in L^1(0,T;\LL)$, we say that $u \in L^1(0,T;H_0^1(\Om))$ is a weak solution to \eqref{equ-u-distri} if $u$ is the $L^1(0,T;H^1(\Om))$-limit of $\{ u(F_n),\ n \in \N \}$, where $\{ F_n,\ n \in \N \} \in L^\infty(0,T;\LL)^\N$ is an approximation sequence of $F$ in $L^1(0,T;\LL)$, i.e.
\begin{equation}
\label{equ-lim}
\lim_{n \to \infty} \norm{u-u(F_n)}_{L^1(0,T;H^1(\Om))}=0,
\end{equation}
where $F_n \in L^\infty(0,T;\LL)$ for all $n \in \N$ and $\lim_{n\to\infty} \norm{F-F_n}_{L^1(0,T;\LL)}=0$.

Notice that the above definition is meaningful
in the sense that, firstly, there exists a unique function $u \in L^1(0,T;H_0^1(\Om))$ obeying \eqref{equ-lim}, and secondly,  $u$ depends only on $F$ and not on the choice of the converging sequence $\{ F_n,\ n \in \N \}$. 
Indeed, the first claim follows from the fact that $\{ u(F_n),\ n \in \N \}$ is a Cauchy sequence in $L^1(0,T;H_0^1(\Om))$. This can be seen from the estimate $\norm{u(F_{n+k})-u(F_n)}_{L^1(0,T;H^1(\Om))} \le C \e^{CT} \norm{u(F_{n+k})-u(F_n)}_{L^1(0,T;\LL)}$ arising for all $n$ and $k$ in $\N$ by applying \eqref{esti_u-h2} with $p=1$ and $\ka=\f{1}{2}$ to the $L^1(0,T;H_0^1(\Om))$-solution $u(F_{n+k})-u(F_n)$ of the system \eqref{equ-u-distri} associated with $u_0=0$ and $F=F_{n+k}-F_n$. Moreover, if $\{\tilde{F}_n,\ n \in \N \}$ is another $L^\infty(0,T;\LL)$-sequence fulfilling $\lim_{n \to \infty} \tilde{F}_n = F$ in $L^1(0,T;\LL)$, then we have
$\norm{u(F_n)-u(\tilde{F}_n)}_{L^1(0,T;H^1(\Om))} \leq C \e^{CT} \norm{F_n-\tilde{F}_n}_{L^1(0,T;\LL)}$ from
\eqref{esti_u-h2} with $p=1$ and $\ka=\f{1}{2}$, since $u(F_n)-u(\tilde{F}_n)$ is a solution to the IBVP \eqref{equ-u-distri} associated with $u_0=0$ and $F=F_n-\tilde{F}_n$.
As a consequence Theorem \ref{thm-forward} with $T \in (0,+\infty)$ and Theorem \ref{thm-L1} remain valid with $F \in L^1(0,T;\LL)$. In particular, we infer from \eqref{esti_u-h2} that the solution $u$ to \eqref{equ-u-distri} associated with $u_0=0$ and $F\in L^1(0,T;\LL)$, fulfills
\begin{equation}
\label{esti_u-h2-b}
\norm{u}_{L^p(0,T;H^{2\ka}(\Om))} \le C \e^{C T} \norm{F}_{L^1(0,T;\LL)},\ \ka \in [0,1),\ p \in \left[1,\frac{1}{1-\al_0(1-\ka)} \right).
\end{equation}
Similarly, it is apparent that the statement of Theorem \ref{thm-H2} still holds for $T \in (0,+\infty)$ and $F \in L^2(Q)$.

\subsection{Brief comments and outline}
To our best knowledge, the only mathematical paper besides this one, dealing with the existence and uniqueness issues for solutions to DO fractional diffusion equations, is \cite{KR}. But the analysis carried out in \cite{KR} is different from the one of the present paper in many aspects. As already mentioned in Section \ref{sec-intro-distri}, the approach of \cite{KR} is variational whereas we study the original function of the solution to the Laplace transform of \eqref{equ-u-distri}.
This allows us to show existence of a unique weak solution to \eqref{equ-u-distri} within the class $\cC((0,T],\LL)$ (and even $\cC([0,T],\LL)$ if $u_0=0$), whereas the solution exhibited in \cite{KR} lies in $L^2(Q)$. Similarly, it is unclear whether the improved regularity estimates \eqref{esti-u}-\eqref{esti_u-h2} or the time analyticity of the weak solution can be derived from the scheme of \cite{KR}. This being said, we stress out once more that the approach of \cite{KR} applies to non-autonomous systems, which is not the case of the analysis presented in this text.

Finally, we point out that Definition \ref{def-u-distri} of a weak solution to \eqref{equ-u-distri} (as the original function of the solution to the Laplace transform with respect to the time variable of this system) is inspired by the analysis carried out in \cite{KSY}, which is concerned with space dependent variable order (VO) time-fractional diffusion equations. It is well known that the weak solution to CO time-fractional diffusion equations can be expressed in terms of Mittag-Leffler functions, see e.g. \cite{FK,KY15}. Nevertheless, such an explicit representation formula is no longer valid for DO or space-dependent VO time-fractional diffusion equations, as the inversion method of the Laplace transform is technically more involved in these two cases. This specific difficulty arising from the non-constancy of the order of DO or VO time-fractional equations is the main difference with the analysis of their CO counterpart. 

The article is organized as follows. 
In Section \ref{sec-sol} we prove Theorem \ref{thm-forward} by showing existence of a unique weak solution to the IBVP \eqref{equ-u-distri}, enjoying a Duhamel-like representation formula. Section \ref{sec-proofs} contains the proof of Theorems \ref{thm-L1} and \ref{thm-H2}, based on careful analysis of the above mentioned representation of the solution. In Section \ref{sec-an} we establish the time analytic property of the solution to \eqref{equ-u-distri}, claimed in Theorem \ref{thm-an} .
Next,  in Section \ref{sec-lip_to_orders}, we prove Theorem \ref{thm-Lip-coef} stating that the weak solution to 
\eqref{equ-u-distri}
depends continuously on the distributed order weight function, the diffusion coefficients and the electric potential. Finally, in the appendix presented in Section \ref{sec-app}, we collect the proof of an auxiliary result used in the derivation of Theorem \ref{thm-H2}.

\section{Representation of the solution: proof of Theorem \ref{thm-forward}}
\label{sec-sol}
The proof of Theorem \ref{thm-forward} is based on an effective representation of the solution to the IBVP \eqref{equ-u-distri}, that is derived in Section \ref{sec-repsol} and presented in Proposition \ref{pr-euws}. As a preamble, several useful properties of the function $s \mapsto w(s)=\int_0^1 s^{\al-1} \mu(\al) \da$, that are needed by the proofs of Proposition \ref{pr-euws} and Theorems \ref{thm-forward} and \ref{thm-H2}, are collected in the coming section.

\subsection{Three auxiliary results on $w$}

We start by lower bounding $\abs{s w(s)+\la}$ with respect to $\la$, uniformly in $s \in \C \setminus \R_-$ and $\lambda \in (0,+\infty)$, where $\R_-:=(-\infty,0]$.

\begin{lem}
\label{lm0}
Let $\mu \in L^\infty(0,1)$ be non-negative. Then, for all $\la \in (0,+\infty)$ and all $s=r \e^{\pm \i \be}$, where $r \in (0,+\infty)$ and $\be \in [0,\pi)$, we have
\begin{equation}
\label{a0b} 
\abs{s w(s)+\la}\ge C_{\be} \la\ \mbox{with}\ C_{\be}:=  \left\{ \begin{array}{cl} 1 & \mbox{if}\ \be \in \left[0, \f{\pi}{2} \right], \\ \f{\sin \be}{2} & \mbox{if}\ \be \in \left( \f{\pi}{2}, \pi \right). \end{array} \right.
\end{equation}
Moreover, if $\be \in \left( \f\pi2,\pi \right)$, then we have in addition:
\begin{equation}
\label{a0} 
\f{\la^\nu \abs{sw(s)}^{1-\nu}}{\abs{sw(s)+\la}}
\le \f2{\sin \be},\ \nu \in [0,1].
\end{equation}
\end{lem}
\begin{proof}
We start with \eqref{a0b}.
The case of $\be \in \left[0, \f{\pi}{2} \right]$ is easily treated, as we have
$$ 
\Re( s w(s) ) = \int_0^1 r^\al  \cos(\al \be) \mu(\al) \da \ge 0, 
$$
since $\mu$ is non-negative, and consequently $\Re( s w(s) + \la ) \ge \la$. 
In order to examine the case where $\be \in \left(\f{\pi}{2}, \pi \right)$, we put 
\begin{equation}
\label{def-fg}
f(s):=\int_0^{\f{\pi}{2 \be}} r^\al \cos(\al \be) \mu(\al) \da 
\mbox{ and } 
g(s):= \int_{\f{\pi}{2\be}}^1 r^\al | \cos(\al \be) | \mu(\al) \da, 
\end{equation}
in such a way that we have
\begin{equation}
\label{a1} 
\Re( s w(s) ) = f(s) - g(s),
\end{equation}
with
\begin{equation}
\label{a2}
f(s) \geq 0\ \mbox{and}\ 0\le g(s)\le\f{\Im(sw(s))}{\sin\be}.
\end{equation}
In the last inequality of \eqref{a2}, we used the fact that $\sin ( \al \be) \ge \sin \be >0$ for all $\al \in \left[ \f{\pi}{2\be}, 1 \right]$ in order to write
$g(s) \le \int_{\f{\pi}{2\be}}^1 r^\al  \mu(\al) \da \le \f{\int_{\f{\pi}{2\be}}^1 r^\al \sin ( \al \be) \mu(\al) \da}{\sin \be}$.
Therefore, if $\Im(sw(s)) \le \f{\la \sin\be}2$ then we get from \eqref{a1}-\eqref{a2} that $\Re(sw(s)+\la)\ge\f{\la}2$, which yields \eqref{a0b}.

We turn now to proving \eqref{a0}. To this end, we infer from \eqref{a1}-\eqref{a2} that
\begin{equation}
\label{a3}
\f{\la^\nu \left( \Im (sw(s)) \right)^{1-\nu}}{\abs{sw(s)+\la}}
\le \left( \f2{\sin\be} \right)^\nu,\ \nu\in[0,1].
\end{equation}
Indeed, we have already noticed that for $\Im(sw(s)) \le\f{\la\sin\be}2$ we have $\Re( sw(s) + \la )\ge\f{\la}2$, whence
$$
\abs{s w(s) + \la} 
\ge \sqrt{ \left( \f{\la}2 \right)^2+ (\Im (s w(s) ))^2}
\geq \left( \f{\la}{2} \right)^\nu \left( \Im (s w(s)) \right)^{1-\nu},\ \nu \in [0,1].
$$
On the other hand, if $\Im(s w(s)) > \f{\la \sin \be}2$ then it holds true that
$$
\abs{s w(s) + \la}
\geq \Im ( s w(s) ) \geq \left( \f{\la \sin \be}2 \right)^\nu \left(\Im (s w(s))\right)^{1-\nu},\ \nu \in [0,1].
$$
Having established \eqref{a3}, we are now in position to prove \eqref{a0}. To do that, we refer once more to \eqref{a1}-\eqref{a2} and examine the two cases $f(s) \geq g(s)$ and $f(s) < g(s)$ separately.
We start with $f(s) \geq g(s)$, involving $\Re ( s w(s) ) \geq 0$ in virtue of \eqref{a1}. Thus we get
$$ 
\abs{s w(s) + \la}
\geq \sqrt{ \la^2 + \abs{s w(s)}^2} \geq \la^\nu | s w(s) |^{1-\nu},\ \nu \in [0,1],
$$
since $\lambda \in (0,+\infty)$, which entails \eqref{a0}.
On the other hand, if $f(s) < g(s)$ then it holds true that
$$
\abs{\Re (s w(s))}= g(s) - f(s) \leq g(s) \leq \f{\Im (s w(s))}{\sin \be}, 
$$
from \eqref{a1}-\eqref{a2}. As a consequence we have
$\abs{s w(s)} \le \sqrt{1 + \left( \f1{\sin \be} \right)^2} \Im (sw(s))$, which, combined with \eqref{a3}, yields \eqref{a0}.
\end{proof}

The second result provides for all $s \in \C \setminus \R_-$ and all $\lambda \in (0,+\infty)$, a suitable lower bound on $\abs{s w(s)+\la}$, expressed in terms of $\abs{s}$.

\begin{lem}
\label{lem-sw>s}
Let $\mu \in L^\infty(0,1;\R_+)$ fulfill \eqref{cnd-mu1}. Then, there exists a constant $C$, depending only on $\de$, $\al_0$ and $\mu$, such that we have
\begin{equation}
\label{ess1} 
\abs{s w(s)+\la} \ge C \min(\abs{s}^{\al_0-\de},\abs{s}^{\al_0}),\ \ s\in\C \setminus \R_-,\ \la \in (0,+\infty).
\end{equation}
\end{lem}
\begin{proof}
Let us first consider the case where $s=re^{\i\be}$ with $r\in(0,+\infty)$ and $\be \in \left[ -\f{\pi}2,\f{\pi}2 \right]$. Since
$\cos(\al\be)\ge0$ for every $\al\in[0,1]$, then we have
$$
\abs{sw(s)+\la}
\ge \int_0^1 r^\al\cos(\al\be)\mu(\al)\da + \la \ge \int_{\al_0-\de}^{\al_0}r^\al \cos(\al\be)\mu(\al)\da, $$
in virtue of \eqref{cnd-mu1}, so we obtain \eqref{ess1} with $C:=\f{\de\mu(\al_0)}{2}\cos \left(\f{\al_0\pi}2\right) $. 
Similarly, if $\abs{\be} \in \left(\f{\pi}2,\pi \right)$ we infer from \eqref{cnd-mu1} that
$$
\abs{sw(s)+\la}
\ge \int_0^1r^\al \abs{\sin(\al \be)} \mu(\al)\da
\ge \int_{\al_0-\de}^{\al_0}r^\al \abs{\sin(\al \be)} \mu(\al)\da,
$$
which yields \eqref{ess1} with $C:=\f{\de \mu(\al_0)}{2} \min \left(\sin \left( \f{(\al_0-\de)\pi}2 \right), \sin(\al_0 \pi)\right)$.
\end{proof}

Finally, we estimate from above the mapping $s \mapsto \abs{w(s)}$ (resp., $s \mapsto \abs{s w(s)}$) by a suitable continuous monotically decreasing (resp., increasing) function of $\abs{s}$.

\begin{lem}
\label{lm-te}
Assume that $\mu \in L^{\infty}(0,1)$ is non-negative.
Then, for all $s \in \C \setminus \R_-$, we have
\begin{equation}
\label{y1}
\abs{s w(s)} \le \| \mu \|_{L^\infty(0,1)} \zeta(\abs{s})\ \mbox{and}\  \abs{w(s)} \le \| \mu \|_{L^\infty(0,1)} \vartheta(\abs{s}),
\end{equation}
where:
\begin{enumerate}[(a)]
\item The mapping
$$ 
\zeta: r \mapsto 
\left\{ 
\begin{array}{cl} 
\f{r-1}{\log r} & \mbox{if}\ r \in (0,1) \cup (1,+\infty), 
\\ 
1 & \mbox{if}\ r=1, 
\end{array} 
\right. 
$$
is continuous and monotonically increasing on $(0,+\infty)$;
\item The function $\vartheta(r)= \f{\zeta(r)}{r}$ is continuous and monotonically decreasing on $(0,+\infty)$.
\end{enumerate}
\end{lem}

\begin{proof}
Since $\abs{s w(s)} \le \int_0^1 \abs{s}^\al \mu(\al) \da \le \| \mu \|_{L^\infty(0,1)}\int_0^1 \abs{s}^\al \da$
from the very definition of $w$, we get \eqref{y1}.
Next, for all  $r \in (0,1) \cup (1,+\infty)$, it holds true that
$\zeta'(r) = \f{h(r)}{r ( \log r )^2}$
with $h(r):= r \log r - r +1$. Further, since $h'(r)=\log r$, we have $h(r) \geq h(1)=0$ for all $r \in (0,1)$, giving $\zeta'(r) \geq 0$ and hence Statement (a). Finally, Statement (b) follows from the identity $\vartheta(r)= \zeta \left( \f{1}{r} \right)$ for all $r \in (0,+\infty)$.
\end{proof}

Armed with the three above lemmas, we may now turn to showing that the IBVP \eqref{equ-u-distri} admits a unique solution enjoying a Duhamel representation formula.

\subsection{A representation formula}
\label{sec-repsol}
For $\ve \in (0,+\infty)$ and $\te \in \left( \f{\pi}{2}, \pi \right)$, we introduce the following contour in $\C$,
\begin{equation}
\label{def-ga} 
\ga(\ve,\te) := \ga_-(\ve,\te) \cup \ga_c(\ve,\te) \cup \ga_+(\ve,\te),
\end{equation}
where
\begin{equation}
\label{def-ga-p} 
\ga_\pm(\ve,\te) 
:= \{ s\in\C,\ \arg s = \pm\te,\ \abs{s} \ge \ve\}\ \mbox{and}\
\ga_c(\ve,\te) := \{ s\in\C,\ \abs{\arg s}\le \te,\ \abs{s}= \ve\}.
\end{equation}
It will prove to be useful for describing the weak solution to the IBVP \eqref{equ-u-distri}, which is the purpose of the following
result.

\begin{prop}
\label{pr-euws}
Let $\mu$, $T$, $u_0$ and $F$ be the same as in Theorem \ref{thm-forward}.
Then,  there exists a unique weak solution  
\begin{equation}
\label{sol-u-distri}
u(t) = S_0(t)u_0 + \int_0^t S_1(t-\tau) F(\tau) \d\tau,\ t \in [0,T],
\end{equation}
to the IBVP  \eqref{equ-u-distri}, where we have set
\begin{equation}
\label{def-S0}
S_0(t) \psi :=
\f1{2\pi\i} \sum_{n=1}^{+\infty} \left( \int_{\ga(\ve,\te)}\f{w(s)}{sw(s)+\la_n}\e^{st}\ds \right)
\langle \psi,\vp_n \rangle_\LL \vp_n,
\end{equation}
and
\begin{equation}
\label{def-S1}
S_1(t) \psi :=
\f1{2\pi\i} \sum_{n=1}^{+\infty} \left( \int_{\ga(\ve,\te)}\f{1}{sw(s)+\la_n}\e^{st}\ds \right)
\langle \psi,\vp_n \rangle_\LL \vp_n,
\end{equation}
for all $\psi \in \LL$, the two above integrals being independent of the choice of $\ve \in (0,+\infty)$ and $\te \in \left( \f{\pi}{2}, \pi \right)$.
\end{prop}

\begin{proof}
The proof is divided into two steps, the first one being concerned with the case of a uniformly vanishing source term $F$, whereas the second one deals with an identically zero initial state $u_0$.\\ 

\noindent 1) {\it Step 1: $F=0$}. In order to characterize $S_0$ we assume that $F=0$ and put
\begin{equation}
\label{def-Y}
Y(s) u_0
:=\sum_{n=1}^\infty \f{w(s)}{s^2(sw(s)+\la_n)} \langle u_0,\vp_n\rangle_\LL \vp_n,\ s\in\C \setminus\R_-. 
\end{equation}
For all $r \in [2,+\infty)$ and $\eta \in \R$, it is clear from \eqref{a0b} and the estimate
$\abs{w(r+\i \eta)} \le \f{\| \mu \|_{L^\infty(0,1)} }{\log r}$, arising from the second inequality of \eqref{y1} and Statement (b) in Lemma \ref{lm-te}, that the inequality
\begin{equation}
\label{ess}
\abs{\langle Y(r+\i\eta)u_0,\vp_n\rangle_\LL}
\le  c \f{\abs{\langle u_0,\vp_n\rangle_\LL}}{\la_n(r^2+\eta^2)}
\le c \f{\abs{\langle u_0,\vp_n\rangle_\LL}}{\la_n(4+\eta^2)},\  n \in\N,
\end{equation}
holds with $c:=\f{\| \mu \|_{L^\infty(0,1)} }{\log 2}$. Therefore we have
\begin{equation}
\label{sup}
\sup_{r \in [2,+\infty)} \norm{Y(r+\i \cdot)u_0}_{L^k(\R,\LL)}  < +\infty,\ k=1,2.
\end{equation}
From \eqref{sup} with $k=1$, it then follows that the function
$$ 
\tilde{y}(t) := \f1{2\pi\i} \int_{-\i\infty}^{+\i\infty} \e^{t s} Y(s+2)u_0 \d p 
= \f1{2 \pi} \int_{-\infty}^{+\infty} \e^{\i t\eta} Y(2 + \i \eta )u_0 \d\eta,
$$
is well defined for each $t \in\R$. Moreover, since the mapping $s \mapsto \e^{t s} Y(s+2) u_0$ is holomorphic in $\C \setminus (-\infty,-2]$, we have
\begin{equation}
\label{bm1} 
\tilde{y}(t)= \f1{2\pi\i} \int_{\rho-\i\infty}^{\rho+\i\infty} \e^{t s} Y(s+2)u_0 \ds,\ \rho \in(0,+\infty).
\end{equation}
Indeed, for all $R \in (1,+\infty)$ and all $\rho \in (0,+\infty)$, the Cauchy formula yields
\begin{equation}
\label{cau}  
\int_{-\i R}^{\i R} \e^{t s} Y(s+2)u_0 \ds = \int_{\rho-\i R}^{\rho+\i R} \e^{t s} Y(s+2)u_0 \ds 
- \int_{I_{R,\rho}^+ \cup I_{R,\rho}^-} \e^{t s} Y(s+2)u_0 \ds, 
\end{equation}
with $I_{R,\rho}^+:=[\i R,\i R+\rho]$ and $I_{R,\rho}^-:= [-\i R+\rho,-\i R]$, and we know from \eqref{ess} that
$$
\norm{\int_{I_{R,\rho}^\pm} \e^{t s} Y(s+2)u_0 ds}_\LL 
\le \norm{\int_{0}^\rho \e^{t (\tau \pm \i R)} Y(\tau+2 \pm \i R)u_0 \d\tau}_\LL
\le C \rho \e^{\rho t}|2 \pm\i R|^{-2}\norm{u_0}_\LL,
$$
so we find \eqref{bm1} by sending $R$ to infinity in \eqref{cau}. 

Further, with reference to \eqref{sup} with $k=1$, we infer from \eqref{bm1} that 
$$ 
\norm{\tilde{y}(t)}_\LL 
=\f1{2 \pi} \norm{ \int_{-\infty}^{+\infty} \e^{t (\rho+ \i\eta)} Y(\rho+2 +\i \eta )u_0 \d\eta}_\LL
\le \f{\e^{\rho t}}{2\pi} \sup_{r \in [2,+\infty)} \norm{ Y(r+\i\cdot)u_0}_{L^1(\R,\LL) },
$$
uniformly in $t\in\R$ and $\rho \in (0,+\infty)$. Then, depending on whether $t \in (-\infty,0)$ or $t \in [0,+\infty)$, we send $\rho$ to either infinity or to zero in the right hand side of the above estimate, and get:
\begin{equation}
\label{ww}
\norm{\tilde{y}}_{L^\infty(\R,\LL)}
\leq\f1{2 \pi} \sup_{r\in [2,+\infty)} \norm{Y(r+\i\cdot)u_0}_{L^1(\R,\LL)}\ \mbox{and}\ \tilde{y}(t)=0\ \mbox{for all}\ t\in(-\infty,0).
\end{equation}
As a consequence we have $\tilde{y}\in L^\infty(\R,\LL)\cap \cS'(\R_+,\LL)$. 

Moreover, in light of \eqref{sup}, we infer from \cite[Theorem 19.2 and the following remark]{Ru} and the holomorphicity of $s \mapsto Y(s+2)$ in $\C_+$, that $\cL[\tilde{y}](s)=Y(s+2)$ for all $s \in \C_+$. Therefore, putting
$$
y(t):= e^{2t}\tilde{y}(t)=\f{\e^{2t}}{2\pi\i} \int_{-\i\infty}^{+\i\infty} \e^{t p} Y(s+2)u_0 \ds=\f1{2\pi\i} \int_{2-\i\infty}^{2+\i\infty} \e^{st} Y(s)u_0 \ds,\ t \in \R, 
$$
we find that
\begin{equation}
\label{lap1}
\cL[y](s)=\cL[\tilde{y}](s-2)=Y(s),\ s\in\{z\in\C;\ \Re z \in (2,+\infty)\}.
\end{equation}
Further, since the mapping $s \mapsto e^{t s} Y(s)a$ is holomorphic in $\C \setminus \R_-$, the Cauchy formula yields 
\begin{equation}
\label{t1b}
y(t)=\f1{2\pi\i} \int_{\ga(\ve,\te)} \e^{t s} Y(s)u_0 \ds,\ t \in \R_+,
\end{equation}
for all $\te \in \left( \f{\pi} 2,\pi \right)$ and all $\ve \in (0,1)$, the contour $\ga(\ve,\te)$ being defined by \eqref{def-ga}-\eqref{def-ga-p}.
Here, we used the fact that
$$
\lim_{\eta \to +\infty} \int_{\eta ( (\tan \te)^{-1} \pm \i )}^{2\pm \i\eta} \e^{t s} Y(s)u_0 \ds =0\ \mbox{in}\ \LL,\ t \in \R_+.
$$ 
This can be easily deduced from the following basic estimate, arising from Lemma \ref{lm0}, the second inequality of \eqref{y1} and the second statement of Lemma \ref{lm-te}, 
\beas
\norm{\int_{\eta ( (\tan \te)^{-1} \pm \i )}^{2 \pm \i \eta} \e^{t s} Y(s)u_0 \ds}_\LL & =  &
\norm{ \int_{\eta (\tan \te)^{-1}}^2 \e^{t (2 \pm \i \eta)} Y(r \pm \i \eta)u_0 \dr }_\LL 
\nonumber\\
& \leq & C \e^{2 t} (2 - \eta (\tan \te)^{-1}) \eta^{-2}\|u_0\|_\LL,\ \eta \in (1,+\infty),
\eeas
where $C$ is positive constant depending only on $\te$ and $\mu$.
We turn now to examining the right hand side of \eqref{t1b}. To this purpose, we write for all $t \in \R_+$,
$$ 
y(t) =\sum_{\ell \in \{ c, \pm \}} y_\ell(t)\ \mbox{with}\ y_\ell(t)
:= \f{1}{2\pi\i} \int_{\ga_\ell(\ve,\te)} \e^{t s} Y(s)u_0 \ds\ \mbox{for}\ \ell \in \{ c,\pm \},
$$
and we infer from \eqref{a0b}, the second inequality of \eqref{y1} and Statement (b) in Lemma \ref{lm-te}, that
\begin{equation}
\label{vpm1}
\norm{y_{\pm}(t)}_\LL \le  \f{\| \mu \|_{L^\infty(0,1)}}{\pi \la_1 \sin \te} \norm{u_0}_\LL \vartheta(\ve) \int_{\ve}^{+\infty} r^{-2} 
\e^{r (\cos \te) t} \d r 
\end{equation}
and
\begin{equation}
\label{vc1}
\norm{y_c(t)}_\LL
\le \f{ \| \mu \|_{L^\infty(0,1)} (2 \te)^{\f{1}{2}}}{\pi \la_1 \sin \te}  \norm{u_0}_\LL \ve^{-\f{3}{2}} \vartheta(\ve)
\e^{\ve t}.
\end{equation}
For $t \in [0,\e]$, we take $\ve=\e^{-1}$ in \eqref{vpm1}-\eqref{vc1} and get
$\norm{y_{\pm}(t)}_\LL + \norm{y_c(t)}_\LL \leq c(\te,\mu) \norm{u_0}_{L^2(\Omega)}$, where $c(\te,\mu)$ denotes a generic positive constant depending only on $\te $ and $\mu$. Similarly, 
for $t \in (\e,+\infty)$, we choose $\ve=t^{-1}$, take into account that $\vartheta(t^{-1})\le t$, and obtain $\norm{y_{\pm}(t)}_\LL + \norm{y_c(t)}_\LL \leq c(\te,\mu) \norm{u_0}_{L^2(\Omega)} t^{\f{5}{2}}$. Thus, setting $\langle t \rangle:=(1+t^2)^{\f{1}{2}}$, 
we find that
$$ 
\norm{y(t)}_\LL 
\le c(\te,\mu) \|u_0\|_\LL \langle t \rangle^{\f{5}{2}},\ t\in\R_+,
$$
which entails $t \mapsto \langle t \rangle^{-\f{5}{2}} y(t) \in L^\infty(\R_+,\LL)$ and hence $y \in \cS'(\R_+,\LL)$, by \eqref{ww}.

Next, as both functions $s \mapsto \cL[y](s)$ and $s \mapsto Y(s)$ are holomorphic in $\C_+$, \eqref{lap1} yields 
$\cL[y](s)=Y(s)$ for all $s\in\C_+$, by unique continuation. As a consequence the Laplace transform of $u :=  \partial_t^2 y \in \cS'(\R_+,\LL)$ reads
$$
\cL[u](s) = s^2 Y(s) 
= \sum_{n=1}^\infty \f{w(s)}{sw(s)+\la_n} \langle u_0,\vp_n \rangle_\LL \vp_n,\ s \in \C_+,
$$
which establishes for each $s \in \C_+$ that $\cL[u](s)$ is a solution to \eqref{eq1} with $F=0$. Moreover, since the operator $A$ is positive and $s w(s) \in [0,+\infty)$ for all $s \in (0,+\infty)$, then $A+s w(s)$ is boundedly invertible in $L^2(\Omega)$. Therefore, $\cL[u](s)=(A+sw(s))^{-1} w(s) u_0$ is uniquely defined by \eqref{eq1} and the same is true for $u$, by injectivity of the Laplace transform. 

It remains to prove that $u$ is expressed by \eqref{sol-u-distri} with $F=0$. This can be done by noticing for
all $s \in \ga(\ve,\te)$ that the function $t \mapsto Y(s) \e^{s t} u_0$ is twice continuously differentiable in $\R_+$ and for all $t \in (0,+\infty)$ that $s \mapsto s^k Y(s) \e^{s t} u_0 \in L^1(\ga(\ve,\te))$ with $k=1,2$. As a matter of fact, we have
$$
\norm{s^k Y(s) e^{s t}u_0}_\LL 
\leq \f{2\|\mu\|_{L^\infty(0,1)}}{\la_1 \sin\te} \vartheta(\ve) |s|^{k-2} \e^{\abs{s} (\cos \te) t}\|u_0\|_\LL,\ k=1,2,\ s \in \ga_\pm(\ve,\te),\ t \in \R_+,$$
from the definition \eqref{def-Y} of $Y$, Lemma \ref{lm0}  and Statement (b) in Lemma \ref{lm-te}, whence $r \mapsto r^k Y(r \e^{\pm i \te})e^{r (\cos \te) t} \in L^1(\ve,+\infty)$.
Therefore, with reference to \eqref{def-S0} we deduce from \eqref{t1b} that $u(t)=S_0(t) u_0$ for all $t \in(0,+\infty)$. This yields $u \in \cC((0,T],\LL)$ and completes the proof of the result when $F=0$. We turn now to examining the case where $u_0=0$.\\

\noindent 2) {\it Step 2: $u_0=0$}. For $s \in \C \setminus \R_-$, we introduce the family of bounded operators in $\LL$,
\begin{equation}
\label{def-Phi}
\Phi(s)\psi := \sum_{n=1}^\infty \f1{s(sw(s)+\la_n)}\langle \psi,\vp_n\rangle_\LL \vp_n,\ \psi\in \LL,
\end{equation}
and we recall from \eqref{ess1} that 
\begin{equation}
\label{ess2}
\norm{\Phi(s)}_{\cB(\LL)}\leq C\max( \abs{s}^{-1-\al_0}, \abs{s}^{-1-\al_0+\de}),\ s \in z \in\C \setminus \R_-,
\end{equation} 
for some positive constant $C$, depending only on $\mu$, in such a way that we have
\begin{equation}
\label{ess3}
\sup_{r \in [1,+\infty) } \norm{\Phi(r+\i \cdot)}_{L^k(\R,\cB(\LL)) } <+\infty,\ k=1,2.
\end{equation}
Thus, by arguing in the same way as in Step 1, we infer from \eqref{ess2}-\eqref{ess3} that
\begin{equation}
\label{s3}
\phi(t):=\f{1}{2\pi\i} \int_{1-\i\infty}^{1+\i\infty} \e^{ts}\Phi(s) \ds = \left\{ \begin{array}{ll} 0 & \mbox{if}\ t\in(-\infty,0), \\ \f1{2\pi\i} \int_{\ga(\ve,\te)} \e^{ts} \Phi(s) \ds & \mbox{if}\ t\in[0,+\infty), \end{array} \right.
\end{equation}
for arbitrary $\ve \in (0,1)$ and $\te \in \left( \frac{\pi}{2}, \pi \right)$, the contour $\gamma(\ve,\te)$ being still defined by \eqref{def-ga}-\eqref{def-ga-p}. Moreover, using  \eqref{ess2}, we find that
\begin{equation}
\label{ess4}
t \mapsto \langle t \rangle^{-m_0} \phi(t)\in L^\infty(\R,\cB(\LL))\ \mbox{with}\ m_0:=1+\al_0.
\end{equation}
As a consequence we have $\phi\in \cS'(\R_+,\cB(\LL))$ and 
\begin{equation}
\label{lap3}
\cL[\phi\psi](s)=\Phi(s)\psi,\ s\in\C_+,\ \psi\in \LL.
\end{equation}
Let us denote by $\wt F$ the function $t \mapsto F(t,\cdot)$ extended by zero in $\R \setminus (0,T)$, and put
\begin{equation}
\label{co}
(\phi*\wt F)(t):=\int_0^t \phi(t-\tau)\wt F(\tau)\d\tau,\ t \in \R.
\end{equation}
Evidently, $\phi*\wt{F}\in \cS'(\R_+,\LL)$ when $T \in (0,+\infty)$, whereas for $T=+\infty$, the assumption $t \mapsto \langle t \rangle^{-m} F(t,\cdot) \in L^\infty(\R_+,\LL)$ yields  
\begin{equation}
\label{esss2}
\norm{(\phi*\wt{F})(t)}_\LL \leq \norm{\langle t \rangle^{-m_0}\phi}_{L^\infty(\R_+,\cB(\LL))} \norm{\langle t \rangle^{-m} F}_{L^\infty(\R_+,\LL)} \langle t \rangle^{1+m+m_0},\ t \in (0,+\infty),
\end{equation}
and hence $\phi*\wt{F} \in \cS'(\R_+,\LL)$ as well.
Further, as we have $\inf\{\ep \in (0,+\infty);\ t\mapsto \e^{-\ep t}\phi(t)\in L^1(0,+\infty;\cB(\LL))\}=0$ by \eqref{ess4} and
$\inf\{\ep \in (0,+\infty); \ t\mapsto \e^{-\ep t}\wt{F}(t)\in L^1(0,+\infty;\LL)\}=0$ from the assumption $t \mapsto \langle t \rangle^{-m} F(t,\cdot) \in L^\infty(\R_+,\LL)$, then it holds true for all $s \in \C_+$ that $\cL[\phi](s)
=\int_0^{+\infty} \phi(t)\e^{-st}\dt$ and $\wh{F}(s):=\cL[\wt{F}](s)=\int_0^{+\infty} \wt{F}(t)\e^{-st}\dt$. From this and \eqref{esss2} it then follows that
$\cL[\phi*\wt{F}](s)
=\cL[\phi](s)\cL[\wt{F}](s)$. As a consequence the function $v:=\pa_t(\phi*\wt{F})\in \cS'(\R_+,\LL)$ fulfills
$$
\cL[v](s) = s\cL[\phi*\wt{F}](s) 
= \sum_{n=1}^{+\infty}\f1{sw(s)+\la_n}\langle \wh F(s),\vp_n\rangle_\LL \vp_n,\ s \in \C_+, 
$$
in virtue of \eqref{lap3}, showing that $\cL[v](s)$ is a solution to \eqref{eq1} for all $s\in(0,+\infty)$. 

It remains to show that
\begin{equation}
\label{fin}
v(t)=\int_0^t S_1(t-\tau)\wt{F}(\tau)\d\tau,\ t \in (0,+\infty).
\end{equation}
To do that, we refer to \eqref{s3} and \eqref{co}, apply Fubini's theorem, and obtain
$$
(\phi*\wt{F})(t)=\f1{2\pi\i}\int_{\ga(\ve,\te)} p(s,t)\ds\ \mbox{with}\ p(s,t):=\int_0^t \e^{(t-\tau)s} \Phi(s)\wt{F}(\tau)\d\tau,\ t \in (0,+\infty).
$$
Therefore, we have
$\pa_t p(s,t)
= \Phi(s)\wt{F}(t)+ s \int_0^t \e^{(t-\tau)s} \Phi(s)\wt{F}(\tau)\d\tau$ for all $s\in\ga(\ve,\te)$ and $t \in (0,+\infty)$, and consequently
\begin{eqnarray*}
\norm{\pa_t p(s,t)}_\LL
& \le & \norm{\Phi(s)}_{\cB(\LL)} \left( \norm{\wt{F}(t)}_\LL +  \abs{s \int_0^t  \e^{\tau s} \d\tau} \norm{\wt{F}}_{L^\infty(0,t+1;\LL)} \right)\\
& \le & \left( 1+\abs{s} \int_0^t \e^{\tau \Re s} \d \tau \right) \norm{\Phi(s)}_{\cB(\LL)}\norm{\wt{F}}_{L^\infty(0,t+1;\LL)}.
\end{eqnarray*}
Putting this with \eqref{ess2}, we get that $\norm{\pa_t p(s,t)}_\LL \le C q(s,t) \norm{\wt{F}}_{L^\infty(0,t+1;\LL)}$ for some positive
constant $C$ depending only on $\mu$, and
$$ q(s,t):=\left\{ \begin{array}{ll} \max(\abs{s}^{-1-\al_0+\de},\abs{s}^{-1-\al_0}) \abs{\cos \te}^{-1} & \mbox{if}\ s\in\ga_\pm(\ve,\te), \\
\max( \ve^{-1-\al_0+\de}, \ve^{-1-\al_0})(1+\e^{t\ve})& \mbox{if}\ s \in \ga_c(\ve,\te). \end{array} \right. $$
As $s \mapsto q(s,t) \in L^1(\ga(\ve,\te))$ for each $t \in (0,+\infty)$, then we have
\begin{eqnarray}
v(t) = \pa_t(\phi*\wt{F})(t) & = & \f1{2\pi\i} \int_{\ga(\ve,\te)} \pa_t p(s,t) \ds \nonumber \\
& = & \f1{2\pi\i} \int_{\ga(\ve,\te)}  \int_0^t  \e^{(t-\tau)s} s \Phi(s) \wt{F}(\tau) d\tau \ds + r(\ve,\te) \wt{F}(t), \label{equ-partialS3}
\end{eqnarray}
with $r(\ve,\te):=\f1{2\pi\i} \int_{\ga(\ve,\te)} \Phi(s) \d s$. The next step is to prove that $r(\ve,\te)=0$. For $R \in (1,+\infty)$, put $C_R(\te):= \{ R \e^{i \be},\ \be \in [-\te,\te] \}$ and 
$\ga_R(\ve,\te):= \{ s \in \ga(\ve,\te),\ \abs{s} \le R \}$. In light of Lemma \ref{lem-sw>s}, we have for all $n \in \N$,
\begin{equation}
\label{equ-gaR}
\int_{\ga_R(\ve,\te)} \f1{s(sw(s) + \la_n)} \ds - \int_{C_R(\te)} \f1{s(sw(s) + \la_n)} \ds=0,
\end{equation}
by the Cauchy theorem, with 
$\abs{\int_{C_R(\te)} \f1{s(sw(s) + \la_n)} \ds}
\le 2C\te R^{-\al_0+\de}$, the constant $C$ being independent of $R$ and $n$. Thus, sending $R$ to infinity in \eqref{equ-gaR}, we get that $\int_{\ga(\ve,\te)} \f1{s(sw(s) + \la_n)} \ds=0$ for every $n \in \N$, whence $\int_{\ga(\ve,\te)} \Phi(s) \d s=0$. From this and  \eqref{equ-partialS3} we get
\eqref{fin} upon applying the Fubini theorem. Hence $u:=v_{|Q}$ is a weak solution to \eqref{eq1} with $u_0=0$, expressed by \eqref{sol-u-distri}, which establishes that $u\in \cC((0,T],\LL)$.

Finally, since a weak solution to \eqref{eq1} is necessarily unique in virtue of Definition \ref{def-u-distri} then the desired result follows readily from Step 1 and Step 2 by invoking the superposition principle.
\end{proof}

Theorem \ref{thm-forward} being a straightforward byproduct of Proposition \ref{pr-euws}, we may now turn to proving Theorems \ref{thm-L1} and \ref{thm-H2}.

\section{Improved regularity: proof of Theorems \ref{thm-L1} and \ref{thm-H2}}
\label{sec-proofs}

The proof of Theorems \ref{thm-L1} and \ref{thm-H2} is by means of suitable time-decay estimates for the operators $S_j$, $j=1,2$, introduced in Proposition \ref{pr-euws}, that are established in the coming section.

\subsection{Time-decay estimates}

We start with $S_0$. 

\begin{lem}
Let the weight function $\mu \in L^\infty(0,1)$ be non-negative.
Then, for all $\ga$ and $\tau$ in $(0,1]$, with $\ga\le\tau$, there exists a positive constant $C$, depending only on $\te$, $\mu$, $\la_1$, $\tau$ and $\ga$, such that the estimate
\label{lem-strichartz-homo}
$$
\|A^\tau S_0(t)\psi\|_\LL
\le C \e^{T} \|\psi\|_{D(A^\ga)} t^{\ga-\tau}
$$
holds uniformly in $t \in (0,T]$ and $\psi \in D(A^\ga)$.
\end{lem}
\begin{proof}
In light of \eqref{def-S0}, it is enough to estimate
\begin{equation}
\label{b0z}
E_n(t):=\f{1}{2\pi\i} \int_{\ga_{(\ve,\te)}} 
\f{w(s)}{sw(s)+\la_n} \e^{st}\ds,
\end{equation}
for all $t\in(0,T)$ and $n \in \N$. With reference to \eqref{def-ga}-\eqref{def-ga-p}, we have
\begin{equation}
\label{b0a} 
E_n(t) = \sum_{j \in \{ c, \pm \}} E_{n,j}(t)\ \mbox{where}\ 
E_{n,j}(t) := \f{1}{2\pi\i} \int_{\ga_j{(\ve,\te)}} 
\f{w(s)}{sw(s)+\la_n} \e^{st}\ds\ \mbox{for}\ j \in \{ c,\pm \}.
\end{equation}
We shall treat each of the three terms $E_{n,j}(t)$, $j \in \{ c,\pm \}$, separately. 

We start with $E_{n,c}(t)$. As a preamble, we choose $\ve \in (0,1)$ so small that
\begin{equation}
\label{b0b}
\ve < \min \left( \eta \la_1, \zeta^{-1}(\eta \la_1) \right),
\end{equation}
where $\eta:= \f{1}{2 \| \mu \|_{L^{\infty}(0,1)}}$ and $\zeta^{-1}$ denotes the function inverse to $\zeta$ on $(0,+\infty)$, whose existence is guaranteed by Statement (a) in Lemma \ref{lm-te}. As a matter of fact we may take $\ve:=\f{1}{2} \min \left( 1, \eta \la_1, \zeta^{-1}(\eta \la_1) \right)$ in such a way that $\ve$ is entirely determined by $\la_1$ and $\| \mu \|_{L^\infty(0,1)}$. In light of  \eqref{y1} and \eqref{b0b}, we have
$\abs{s w(s) + \la_n} \geq \f{\la_n}{2}$ for all $s \in \ga_{c}(\ve,\te)$, and consequently
\begin{equation}
\label{b1}
\abs{E_{n,c}(t)} \le \f{\e^{T} \la_1}{\la_n},\ t \in[0,T],\ n\in\N. 
\end{equation}

We turn now to estimating $E_{n,\pm}(t)$. Bearing in mind that $\ga_\pm(\te,\ve)=\{ r \e^{\pm \i \te},\ r \in (\ve,+\infty) \}$, we decompose $E_{n,\pm}$ into the sum
\begin{equation}
\label{b2}
E_{n,\pm}(t) = \sum_{k=1,2} E_{n,\pm,k}(t),
\end{equation}
where
$
E_{n,\pm,k}(t) := \f{1}{2 \pi \i}\int_{I_{n,k}}
\f{\e^{\pm \i \te} w(r \e^{\pm \i \te})}{r \e^{\pm \i \te}w(r \e^{\pm \i \te})+\la_n} \e^{rt \e^{\pm \i \te}} \dr
$ 
for $k=1,2$, $I_{n,1}:=(\ve,\eta \la_n)$ and $I_{n,2}:=(\eta \la_n,+\infty)$.
Notice from \eqref{b0b} that $\ve \in (0,\eta \la_n)$ and hence that the interval $I_{n,1}$ is non-empty.
Next, for all $r \in I_{n,1}$, it is clear from \eqref{a0b} that 
$\abs{r \e^{\pm \i \te} w(r \e^{\pm \i \te}) + \la_n} \ge \f{\la_n \sin \te}{2}$, and from the second inequality of \eqref{y1} combined with Lemma \ref{lm-te}, that
$$ 
\abs{w(r \e^{\pm \i \te})}
\le \| \mu \|_{L^\infty(0,1)} \vartheta(r) 
\le \| \mu \|_{L^\infty(0,1)} \vartheta(\ve) 
\le 
\| \mu \|_{L^\infty(0,1)} \f{\zeta(\ve)}{\ve} 
\le \| \mu \|_{L^\infty(0,1)} \f{\eta \la_1}{\ve} 
\le \f{\la_1}{2\ve}.
$$
As a consequence we have
$$
\abs{E_{n,\pm,1}(t)} 
\le \f{1}{2 \pi}\int_\ve^{\eta \la_n}
\f{|\e^{\pm \i \te} w(r \e^{\pm \i \te})|}{|r \e^{\pm \i \te}w(r \e^{\pm \i \te})+\la_n|} \e^{-rt \abs{\cos\te}} \dr 
\le \f{\la_1}{2 \pi \ve  \la_n \sin \te} \int_\ve^{\eta \la_n}\e^{-rt \abs{\cos\te}} \dr
$$
whence
\begin{equation}
\label{b3}
\abs{E_{n,\pm,1}(t)} \le \f{C}{\la_n t^{\tau-\ga}} \int_\ve^{\eta\la_n} 
\f{\dr}{r^{\tau-\ga}} 
\le \f{C}{\la_n^{\tau-\ga} t^{\tau-\ga}}.
\end{equation}
Here and in the remaining part of this proof, $C$ denotes a positive constant depending only on $\te$, $\norm{\mu}_{L^\infty(0,1)}$, $\la_1$, $\tau$ and $\ga$, that may change from line to line.
Further, applying \eqref{a0} with $\nu=0$, we get that
$$
\abs{E_{n,\pm,2}(t)}
\le \f{1}{2 \pi}\int_{\eta \la_n}^{+\infty}
\f{|r \e^{\pm \i \te} w(r \e^{\pm \i \te})|}{|r \e^{\pm \i \te}w(r \e^{\pm \i \te})+\la_n|} \e^{-rt|\cos\te|} \f{\dr}{r} 
\le \f{1}{\pi \sin \te} \int_{\eta \la_n}^{+\infty} \e^{-rt \abs{\cos\te}} \f{\dr}{r},
$$
which entails
\begin{equation}
\label{b3b}
\abs{E_{n,\pm,2}(t)} \le \f{C}{t^{\tau-\ga}} \int_{\eta \la_n}^{+\infty}  \f{\dr}{r^{1+\tau-\ga}} 
\le \f{C}{\la_n^{\tau-\ga} t^{\tau -\ga}}.
\end{equation}
Now, putting \eqref{b0a} and \eqref{b1}--\eqref{b3b} together, we find that
 $$
\label{b4}
\la_n^{\tau} \abs{E_n(t)}
\le C \e^{T} \la_n^\ga t^{\ga-\tau},\ t \in (0,T],
$$
Therefore, for all $\psi\in D(A^\ga)$ and all $t \in (0,T]$, we have
$$
\|A^\tau S_0(t)\psi\|_\LL^2
=\sum_{n=1}^\infty 
\la_n^{2\tau} \abs{E_n(t)}^2 \abs{\langle \psi,\vp_n \rangle_\LL}^2
\le C^2 \e^{2T} t^{2(\ga-\tau)}
\sum_{n=1}^\infty
\la_n^{2\ga} \abs{\langle \psi,\vp_n \rangle_\LL}^2,
$$
which yields the desired result.
\end{proof}

We turn now to examining the time-evolution of the operator $S_1$.

\begin{lem}
\label{lem-strichartz-nonhomo}
Assume that $\mu\in L^\infty(0,1)$ is non-negative and satisfies \eqref{cnd-mu1}. Then, for all $\ka\in[0,1)$ and all $\be \in (1-\al_0(1-\ka),1)$, there exists a positive constant $C$, depending only on $\te$, $\la_1$, $\mu$, $\ka$ and $\be$, such that the estimate
$$ 
\|A^\ka S_1(t)\psi\|_\LL \le C \e^{T} \|\psi\|_\LL t^{-\be},
$$
holds uniformly in $t\in(0,T]$ and $\psi\in\LL$.
\end{lem}

\begin{proof}
We proceed as in the proof of Lemma \ref{lem-strichartz-homo}, the main novelty being that the function $t \in (0,T] \mapsto E_n(t)$, defined by \eqref{b0z} for all $n \in \N$, is replaced by
\begin{equation}
\label{c0}
G_n(t):=\f{1}{2\pi\i} \int_{\ga_{(\ve,\te)}} 
\f{1}{sw(s)+\la_n} \e^{st}\ds.
\end{equation}
With reference to \eqref{def-ga}-\eqref{def-ga-p}, we write
\begin{equation}
\label{c1} 
G_n(t) = \sum_{j \in \{ c, \pm \}} G_{n,j}(t)\ \mbox{with}\ 
G_{n,j}(t) := \f{1}{2\pi\i} \int_{\ga_j{(\ve,\te)}} 
\f{1}{sw(s)+\la_n} \e^{st}\ds\ \mbox{for}\ j \in \{ c,\pm \},\ \ve \in (0,1),
\end{equation}
and set $\ve=\f{1}{2} \min \left( 1, \eta \la_1, \zeta^{-1}(\eta \la_1) \right)$, where $\eta=\f{1}{2\| \mu \|_{L^\infty(0,1)}}$, as in the proof of Lemma \ref{lem-strichartz-homo}. Thus, using \eqref{y1} and \eqref{b0b}, we find upon arguing as in the derivation of \eqref{b1}, that
\begin{equation}
\label{c2} 
\abs{G_{n,c}(t)} \le\f{2 \e^T}{\la_n},\ t \in (0,T],\ n\in\N.
\end{equation}
The rest of the proof is to estimate $G_{n,\pm}(t)$. Firstly, we split $G_{n,\pm}(t)$ into the sum
\begin{equation}
\label{c3}
G_{n,\pm}(t) = \sum_{k=1,2} G_{n,\pm,k}(t),
\end{equation} 
where
$G_{n,\pm,k}(t) := \f{1}{2 \pi \i}\int_{I_{n,k}} \f{\e^{\pm \i \te}}{r \e^{\pm \i \te} w(r \e^{\pm \i \te}) + \la_n} \e^{r \e^{\pm \i \te} t} \dr$ for $k=1,2$, $I_{n,1}=(\ve,\eta \la_n)$ and $I_{n,2}=(\eta \la_n,+\infty)$. By mimicking the
proof of \eqref{b3}, we obtain for all $t \in (0,T]$ and all $n \in \N$, that
\begin{equation}
\label{c4}
\abs{G_{n,\pm,1}(t)} \le \f{1}{\pi \la_n\sin\te} \int_\ve^{\eta \la_n} \e^{-r | \cos \te | t} \dr \le \f{C}{\la_n^\ka t^\ka}, 
\end{equation}
where $C$ is a positive constant depending only on $\te$, $\| \mu \|_{L^\infty(0,1)}$ and $\ka$. Next, we estimate $G_{n,\pm,2}(t)$ by applying \eqref{a0} with $\nu=\ka$,
\begin{eqnarray}
\abs{G_{n,\pm,2}(t)} & \le & \f{1}{2 \pi \la_n^\ka} \int_{\eta \la_n}^{+\infty} 
\f{\la_n^\ka |r \e^{\pm \i \te} w(r \e^{\pm \i \te})|^{1-\ka}}{|r \e^{\pm \i \te} w(r \e^{\pm \i \te} \e^{\pm \i \te}) + \la_n|} 
\f{\e^{-r | \cos \te | t}}{|r \e^{\pm \i \te} w(r \e^{\pm \i \te})|^{1-\ka}} \dr \nonumber \\
& \le & \f{1}{\la_n^\ka \sin \te} \int_{\eta \la_n}^{+\infty} 
\f{\e^{-r |\cos \te | t}}{|r \e^{\pm \i \te} w(r \e^{\pm \i \te})|^{1-\ka}} \dr, \label{c5}
\end{eqnarray}
and bounding from below with the aid of \eqref{cnd-mu1} the denominator of the integrand in the last integral. Indeed, we have $\abs{r \e^{\pm \i \te} w(r \e^{\pm \i \te})} \ge | \Im ( r \e^{\pm \i \te} w(r \e^{\pm \i \te}) )| \ge \int_0^1 r^\al \sin (\al \te) \mu(\al) \da$ for all $r \in (\eta \la_n , +\infty)$, hence
$\abs{r \e^{\pm \i \te} w(r \e^{\pm \i \te})} \ge  \min(\sin((\al_0-\de)\te), \sin (\al_0 \te)) \f{\mu(\al_0)}{2}  \int_{\al_0-\de}^{\al_0} r^\al \da$. Consequently there exists a positive constant $c_0$, depending only on $\te$, $\mu$ and $\rho$, such that 
\begin{equation}
\label{c5b}
\abs{r \e^{\pm \i \te} w(r \e^{\pm \i \te})} \ge c_0 r^\rho,\ \rho \in (\al_0-\de,\al_0).
\end{equation}
Next, applying \eqref{c5b} with $\rho \in \left( \max \left( \al_0-\de , \f{1-\be}{1-\ka} \right), \al_0 \right)$ for some fixed $\be \in (1-\al_0(1-\ka),1)$, we infer from \eqref{c5} that
\begin{equation}
\label{c6}
\abs{G_{n,\pm,2}(t)} \le \f{C}{\la_n^\ka t^\be} \int_{\eta \la_n}^{+\infty} 
\f{\d r}{r^{\be + \rho(1-\ka)}} \le \f{C}{\la_n^\ka t^\be},
\end{equation}
where $C$ is a positive constant depending only on $\te$, $\la_1$, $\mu$, $\ka$ and $\be$. 
Now, putting \eqref{c1}--\eqref{c4} and \eqref{c6} together, we obtain that
\begin{equation}
\label{esti-Gn}
\abs{G_n(t)}
\le \f{C \e^T}{\la_n^\ka t^{\be}},\ t \in (0,T],\ n \in \N.
\end{equation}
Finally, recalling \eqref{def-S1}, \eqref{c0} and \eqref{esti-Gn}, we end up getting for all $\psi \in \LL$ that
$$
\|A^\ka S_1(t)\psi\|_\LL^2
=\sum_{n=1}^\infty 
\la_n^{2\ka} \abs{G_n(t)}^2 \abs{\langle \psi,\vp_n \rangle_\LL}^2
\le C^2 \e^{2T} t^{-2\be} \sum_{n=1}^\infty \abs{\langle \psi,\vp_n \rangle_\LL}^2,
$$
and the desired result follows from this and the Parseval identity.
\end{proof}

Having established Lemmas \ref{lem-strichartz-homo} and \ref{lem-strichartz-nonhomo}, we may now prove Theorem \ref{thm-L1}.

\subsection{Proof of Theorem \ref{thm-L1}}

We examine the two cases $F=0$ and $u_0=0$ separately.\\

\noindent a) We first assume that $a\in D(A^\ga)$ for some $\ga\in(0,1]$, and that
$F=0$. Then, by applying Lemma \ref{lem-strichartz-homo} with $\tau=1$, we get that
$$
\|A S_0(t)u_0\|_\LL\le C \e^T \|A^\ga u_0 \|_\LL t^{\ga-1},\ t \in (0,T].
$$
This and \eqref{sol-u-distri} yield \eqref{esti-u} in virtue of the equivalence of the norms in $\HH$ and in $D(A)$.

Let us now prove \eqref{esti-u_t}. We stick with the notations used in the proof of Lemmas \ref{lem-strichartz-homo} and \ref{lem-strichartz-nonhomo} and establish for every $\be \in \left( 1-\al_0 \ga , 1 \right)$ that
\begin{equation}
\label{p1}
\left|\f{\d E_n(t)}{\dt}\right|
\le C \e^T \la_n^\ga t^{-\be},\ t \in (0,T],\ n \in \N,
\end{equation}
where $C$ is the constant appearing in \eqref{esti-Gn}.
This can be done with the aid of \eqref{b0z}, \eqref{c0} and the Cauchy theorem, involving
$$
\f{\d E_n(t)}{\dt} 
=\f{1}{2\pi \i}\int_{\ga(\ve,\te)} \f{sw(s)}{sw(s) + \la_n}
\e^{st}\ds
=-\f{1}{2\pi \i}\int_{\ga(\ve,\te)} \f{\la_n}{sw(s) + \la_n}
\e^{st}\ds
=-\la_n G_n(t),
$$
upon applying the estimate \eqref{esti-Gn} with $\ka=1-\ga$.
In light of \eqref{sol-u-distri}-\eqref{def-S0} and \eqref{p1}, we find that
$$
\|\pa_tu(t)\|_\LL^2 
=\sum_{n=1}^\infty \abs{\f{\d E_n(t)}{\dt}}^2 \abs{ \langle u_0,\vp_n \rangle_\LL}^2 \le C^2 \e^{2T} t^{-2 \be} \sum_{n=1}^\infty  \la_n^{2\ga} \abs{ \langle u_0,\vp_n \rangle_\LL}^2,
$$
which leads to \eqref{esti-u_t}.\\

\noindent b) Suppose that $F \in L^1(0,T;\LL)$ and $u_0=0$. We recall for all $\ka \in [0,1)$ and $\be \in (1-\al_0(1-\ka),1)$ from Lemma~\ref{lem-strichartz-nonhomo}, that the estimate
$$
\|A^\ka S_1(t-\tau) F(\tau) \|_\LL
\le C \e^T (t-\tau)^{-\be}\|F(\tau)\|_\LL,\ t \in (0,T],\ \tau \in [0,t),
$$
holds for some constant $C \in (0,+\infty)$, independent of $T$, $t$ and $\tau$. Therefore we have
$$ 
\| A^\ka u (t) \|_\LL 
\leq C \e^T \left( \tau^{-\be} \ast \| F(\tau) \|_\LL \right)(t),\ t \in (0,T], 
$$
by \eqref{sol-u-distri}, where $\ast$ stands for the convolution operator. An application of Young's inequality then yields
$$
\left(\int_0^T \| u(t)\|_{D(A^\ka)}^p \dt\right)^{\f1p}
\le C \e^T \left(\int_0^T t^{-\be p}\dt\right)^{\f1p}
\int_0^T \|F(t) \|_\LL\dt,\
p \in \left[ 1, \f{1}{\be} \right],$$
and \eqref{esti_u-h2} follows from this and the equivalence of the norms in $H^{2 \ka}(\Om)$ and in $D(A^\ka)$.

\subsection{Proof of Theorem \ref{thm-H2}}

The derivation of Theorem \ref{thm-H2} essentially relies on the following technical result, whose proof is postponed to the appendix.

\begin{lem}
\label{lem-G_n}
Assume that $\mu \in \cC([0,1],\R_+)$ fulfills the condition \eqref{cnd-mu1} and let $G_n$, for $n \in \N$, be defined by \eqref{c0}. 
Then, we have
\begin{equation}
\label{g0a}
G_n(t)= -\f 1 \pi\int_0^{+\infty} \Phi_n(r) \e^{-rt} \dr,\ t \in (0,T),
\end{equation}
with
\begin{equation}
\label{g0b}
\Phi_n(r) := \f{\int_0^1r^\al \sin(\pi\al) \mu(\al)\da}
{\left( \int_0^1r^\al \cos(\pi\al) \mu(\al)\da + \la_n \right)^2 + \left( \int_0^1 r^\al \sin(\pi\al) \mu(\al)\da \right)^2}. 
\end{equation}
Moreover, there exists a positive constant $C$, such that we have
\begin{equation}
\label{g0c}
\int_0^{+\infty} \f{\Phi_n(r)}{r} \d r \leq \f{C}{\la_n},\ n \in \N.
\end{equation}
\end{lem}

As a matter of fact we know from \eqref{sol-u-distri} that
$$
\norm{u(t)}_{D(A)}^2
=\norm{\int_0^t A S_1(t-\tau)F(\tau)\d\tau}_\LL^2
=  \sum_{n=1}^\infty \la_n^2 \abs{\int_0^t  G_n(t-\tau) \langle F(\tau),\vp_n \rangle_\LL \d\tau}^2,\ t \in (0,T),
$$
and for each $n \in \N$ we have
\begin{eqnarray*}
\int_0^T \abs{\int_0^t G_n(t-\tau) \langle F(\tau),\vp_n \rangle_\LL \d\tau}^2 \d t & = & \norm{G_n \ast \langle F(\cdot),\vp_n \rangle_\LL}_{L^2(0,T)}^2 \\
&  \leq & \norm{G_n}_{L^1(0,T)}^2 \norm{ \langle F(\cdot),\vp_n \rangle_\LL}_{L^2(0,T)}^2,
\end{eqnarray*}
by Young's inequality, whence
\begin{equation}
\label{f2}
\|u\|_{L^2(0,T;D(A))}^2
\le \sum_{n=1}^\infty \la_n^2 \norm{G_n}_{L^1(0,T)}^2 \norm{\langle F(\cdot),\vp_n \rangle_\LL}_{L^2(0,T)}^2.
\end{equation}
Thus, we are left with the task of estimating $\|  G_n \|_{L^1(0,T)}$. This can be done by combining Lemma \ref{lem-G_n} with Fubini's theorem, in the same way as in the derivation of \cite[Theorem 2.1]{LLY15}. We get for every $n \in \N$ that
$$
\|  G_n \|_{L^1(0,T)}
=\f{1}{\pi} \int_0^{+\infty} \Phi_n(r) \left( \int_0^T \e^{-rt} \dt \right) \dr
\le \f{1}{\pi} \int_0^{+\infty} \f{\Phi_n(r)}{r}\dr
\le \f{C}{\la_n},
$$
the constant $C$ being independent of $n$.
This and \eqref{f2} yield
$$
\norm{u}_{L^2(0,T;D(A)}^2
\le C \sum_{n=1}^{+\infty} \int_0^T \abs{\langle F(t),\vp_n \rangle_\LL}^2 \d t
\le C \norm{F}_{L^2(Q)}^2,
$$
with the aid of the Parseval identity, which establishes the result.



\section{Time analyticity: proof of Theorem \ref{thm-an}}
\label{sec-an}

With reference to Duhamel's representation formula \eqref{sol-u-distri} of the weak solution to \eqref{equ-u-distri}, 
we may treat the two cases $F=0$ and $u_0=0$ separately. We stick with the notations of Section \ref{sec-repsol} and for any open connected set $\cO$ in $\R$ or $\C$, we denote by $\A(\cO,\LL)$ the space of $\LL$-valued analytic functions in $\cO$.\\

\noindent 1) {\it Step 1: $F=0$}.
In light of Proposition \ref{pr-euws}, it is enough to show that the mapping $t \in (0,+\infty) \mapsto S_0(t) u_0$, where $S_0$ is defined by \eqref{def-S0}, is extendable to an analytic map of  $\cO$ into $\LL$, where
$\cO :=\{r \e^{i\omega};\ r \in (0,+\infty),\ \omega \in(-\theta_1,\theta_1) \}$ for some arbitrary $\te_1 \in \left( 0,\te- \f{\pi}{2} \right) \cap (0,\pi-\te)$. 

We start by noticing that $\abs{\pm \te + \om} \in (\te-\te_1,\te+\te_1) \subset \left( \f{\pi}{2} , \pi \right)$ for all $\omega \in (-\te_1,\te_1)$ and hence that $\cos( \pm \theta +\omega) \leq \cos(\theta-\theta_1)$. Thus we have $\abs{\e^{z s}} \le \e^{\abs{z} \abs{s} \cos(\theta-\theta_1)}$ for all $z \in \cO$ and all $s \in \ga_\pm(\ve,\te)$, by \eqref{def-ga-p}, and \eqref{def-Y} then yields $\norm{\e^{z s} s^2 Y(s) u_0}_{L^2(\Omega)} 
\le \frac{2 \norm{\mu}_{L^\infty(0,1)}}{\la_1 \sin \te} \vartheta(\ve) \e^{\abs{z} \abs{s} \cos(\te-\te_1)}$ with the aid of Lemma \ref{lm0} and Statement (b) in Lemma \ref{lm-te}. Thus, taking into account that $\cos(\te-\te_1) \in (-1,0)$, we get for any compact set $K \subset \cO$ that
\begin{equation}
\label{an0}
\norm{\e^{z s} s^2 Y(s) u_0}_{\LL} 
\le \frac{2 \norm{\mu}_{L^\infty(0,1)}}{\la_1 \sin \te} \vartheta(\ve) \e^{d_K \abs{s} \cos(\te-\te_1)} 
,\ z \in K,\ s \in \ga_\pm(\ve,\te),
\end{equation}
where $d_K:= \min \{ \abs{z},\ z \in K \} \in (0,+\infty)$. Next, since $z \mapsto e^{z s} s^2 Y(s) u_0 \in \A(\cO,\LL)$ for all $s \in \ga_\pm(\ve,\te)$, we derive from \eqref{an0} that $z \mapsto \frac{1}{2\pi\i}\int_{\ga_\pm(\ve,\te)} \e^{z s} s^2Y(s) u_0 \d s \in \A(\cO,\LL)$.
Further, since the path $\ga_c(\ve,\te)$ is finitely extended, it is easy to see that $z \mapsto \frac{1}{2 \pi\i} \int_{\ga_c(\ve,\te)} \e^{z s} s^2Y(s) u_0 \d s \in \A(\cO,\LL)$. Therefore, using \eqref{def-ga} we find that
$$z \mapsto \frac{1}{2 \pi\i} \int_{\ga(\ve,\te)} \e^{z s} s^2Y(s) u_0 \d s \in \A(\cO,\LL). $$ 
The desired result follows from this, \eqref{def-S0} and \eqref{def-Y}.\\

\noindent 2) {\it Step 2: $u_0=0$}. Fix $t_0 \in (0,+\infty)$. We shall prove existence of $r \in (0,+\infty)$ such that $t \mapsto \int_0^t S_1(\tau) F(t-\tau) \d \tau$ is extendable to a holomorphic function of $D(t_0,r)$ into $\LL$.
Here, $S_1$ is the operator defined by \eqref{def-S1}, and for all $z_0 \in \C$ and all $R \in (0,+\infty)$, we denote by $D(z_0,R):=\{ z  \in \C;\ \abs{z-z_0} \in [0,R) \}$ the open disk of $\C$ centered at $z_0$ with radius $R$.

First, since $F$ is extendable to a holomorphic function of $S_\rho$ into $\LL$, we may assume without loss of generality that $F \in \A(S_\rho,\LL)$. Next we pick $\de_1 \in \left( 0 , \f{\rho}{4} \right) \cap (0, t_0)$ 
in such a way that $z-\tau \in D(0,\rho) \subset S_\rho$ for all $(z,\tau) \in D(t_0,2\de_1) \times D(t_0,2 \de_1)$ and obtain that the mapping
\begin{equation}
\label{an1}
\tau \mapsto s \Phi(s) \e^{\tau s} F(z-\tau,\cdot) \in \A(D(t_0,2 \de_1),\LL),\ z \in D(t_0,2\de_1), 
\end{equation}
where $\Phi$ the same as in \eqref{def-Phi}.

Further, for $t_1 \in (t_0-\de_1, t_0)$ fixed, since $[t_1,z] \subset D(t_0, 2 \de_1)$ for all $z \in D(t_0,2\de_1)$, we see from \eqref{an1} that
\begin{equation}
\label{an2}
U_*(s,z):=\int_{[t_1,z]}s\Phi(s) \e^{\tau s} F(z-\tau,\cdot) \d\tau,\ s \in \C\setminus\R_-,\ z \in D(t_0,2\delta_1)
\end{equation}
is well defined. Moreover, for all $z \in D(t_0,\de_1)$ and all $h \in D(0,\de_1) \setminus \{ 0 \}$, we have
\begin{eqnarray*}
& & \f{U_*(s,z+h)-U_*(s,z)}{h} \\
& = & \f{1}{h} \left(\int_{[t_1,z+h]} s\Phi(s) \e^{\tau s}F(z+h-\tau,\cdot)d\tau-\int_{[t_1,z]}s\Phi(s) \e^{\tau s}F(z+h-\tau,\cdot) \d\tau\right) \\
& &+\int_{[t_1,z]} s\Phi(s) \e^{\tau s}\f{F(z+h-\tau,\cdot)-F(z-\tau,\cdot)}{h} \d\tau \\
& = & \int_{[z,z+h]}s\Phi(s) \e^{\tau s}F(z+h-\tau,\cdot) \d\tau + \int_{[t_1,z]} s \Phi(s) \e^{\tau s} \f{F(z+h-\tau,\cdot)-F(z-\tau,\cdot)}{h} \d\tau,
\end{eqnarray*} 
and hence
$$\lim_{h \to 0} \frac{U_*(s,z+h)-U_*(s,z)}{h}=s\Phi(s) \e^{z s}F(0,\cdot)+\int_{[t_1,z]}s\Phi(s) \e^{\tau s} \partial_z F(z-\tau,\cdot)d\tau$$
in $\LL$. As a consequence $z \mapsto U_*(s,z)\in \A(D(t_0,\de_1),\LL)$ for all $s \in \C \setminus \R_-.$

Furthermore, using that $[t_1,z] \subset D(t_0,\de_1) \subset \cO$ for each $z \in D(t_0,\de_1)$, we derive from \eqref{an2} that
$$
 \norm{U_*(s,z)}_{\LL} \le 2\de_1\abs{s} \norm{\Phi(s)}_{\cB(\LL)} \norm{F}_{L^\infty(D(0,2\de_1),\LL)}
 \e^{d(t_0,\de_1) \abs{s} \cos(\te\pm\te_1)},\ s \in \ga_\pm(\ve,\te),
$$
with $d(t_0,\de_1):=\inf \{ \abs{z};\ z \in D(t_0,\de_1) \} \in (0,+\infty)$. Thus, bearing in mind that $\te\pm \te_1 \in \left( \f{\pi}{2},\pi \right)$, we infer from this and from \eqref{ess2} that
$$ \norm{U_*(s,z)}_{\LL} \le C \abs{s}^{-1-\al_0+\de},\  z\in D(t_0,\de_1),\ s \in \ga_\pm(\ve,\te), $$
for some positive constant $C$ which is independent of $s$ and $z$. Therefore $z \mapsto \int_{\ga_\pm(\ve,\te)} U_*(s,z) \d s \in \A(D(t_0,\de_1),\LL)$ and hence 
\begin{equation}
\label{an2b}
z \mapsto \int_{\ga(\ve,\te)} U_*(s,z) \d s \in \A(D(t_0,\de_1),\LL).
\end{equation}
 
Similarly, by setting
\begin{equation}
\label{an3}
U_\sharp(s,z) :=\int_{[0,t_1]} s\Phi(s) \e^{\tau s} F(z-\tau) \d\tau,\  s \in \C\setminus\R_-,\ z \in D(t_0,\de_1), 
\end{equation}
and arguing as above, we find that 
\begin{equation}
\label{an3b}
z \mapsto \int_{\ga(\ve,\te)} U_\sharp(s,z) \d s \in \A(D(t_0,\de_1),\LL).
\end{equation}

Finally, putting \eqref{sol-u-distri}, \eqref{def-S1}, \eqref{def-Phi}, \eqref{an2} and \eqref{an3} together, we obtain that the solution $u$ to \eqref{equ-u-distri} reads
$$
u(t,\cdot) = \int_0^t S_1(s) F(t-s)ds = \f{1}{2 \pi \i} \left( \int_{\ga(\ve,\te)} U_*(s,t) \d s+ \int_{\ga(\ve,\te)} U_\sharp(s,t) \d s \right),\ t\in(t_0-\de_1,t_0+\de_1),
$$
and consequently $u\in \A((t_0-\de_1,t_0+\de_1),\LL)$ by \eqref{an2b} and \eqref{an3b}. Since $t_0$ is arbitrary in $(0,+\infty)$, this entails that $u\in \A((0,+\infty),\LL)$, completing the proof of Theorem \ref{thm-an}.

\section{Lipschitz stability: proof of Theorem \ref{thm-Lip-coef}}
\label{sec-lip_to_orders}

The strategy of the proof of the stability inequality \eqref{eq-Lip-coef} essentially follows the lines of the derivation of \cite[Theorem 2.3]{LLY-AMC} and boils down to the estimates \eqref{esti-u}-\eqref{esti-u_t} established in Theorem \ref{thm-forward}. 

For the sake of notational simplicity, we start by rewriting the IBVP \eqref{equ-u-distri} with $F=0$ as
\begin{equation}
\label{eq-gov-u1}
\begin{cases}
\!\begin{alignedat}{2}
&\D_t^{(\mu)}u=L_{a,q} u 
&\quad &\mbox{in}\ Q,
\\
&u(0,\cdot)=u_0 & \quad & \mbox{in}\ \Om,\\
&u=0 & \quad & \mbox{on}\ \Sg,
\end{alignedat}
\end{cases}
\end{equation}
where $L_{a,q} u(x,t):=\mathrm{div}(a(x)\na u(x,t)) + q(x) u(x,t)$ is associated with
the diffusion matrix $a$ and the electric potential $q$.
Next, we notice that $v:=u-\wt u$ is solution to
$$
\begin{cases}
\!\begin{alignedat}{2}
& \D_t^{(\mu)} v=L_{a,q} v + F & \quad & \mbox{in}\ Q,
\\
& v(0,\cdot)=0 & \quad & \mbox{in}\ \Om,
\\
& v=0 & \quad & \mbox{on}\ \Sg,
\end{alignedat}
\end{cases}
$$
with
\begin{equation}
\label{z0}
F:=\D_t^{(\wt\mu-\mu)} \wt u + L_{a -\wt a,q - \wt q} \wt u.
\end{equation} 
Thus, in light of \eqref{esti_u-h2-b}, we are left with the task of estimating the $L^1(0,T;\LL)$-norm of $F$. This will be done with the help of \eqref{esti-u}-\eqref{esti-u_t}, giving
\begin{equation}
\label{eq-est-u1}
\|\wt u\|_{L^1(0,T;\HH)}\le C \e^{CT} T^\ga\ \mbox{and}\
\|\pa_t \wt u\|_{L^1(0,T;\LL)}\le C e^{CT} T^{\al_0\ga}.
\end{equation}
Here and in the remaining part of this proof, $C$ denotes a generic positive constant depending only on $M$, $c_a$, $\ga$, $\al_0$ and $\de$, which may change from line to line. Further, using the fact
that $(a,q)$ and $(\wt a,\wt q)$ are in $C^1(\ov\Om) \times C^0(\ov\Om)$, we obtain that 
\begin{eqnarray}
\|L_{\wt a-a, \wt q - q} \wt u\|_{L^1(0,T;\LL)}
& \le &  \left( \norm{a-\wt a}_{C^1(\ov\Om)} + \norm{q-\wt q}_{C^0(\ov\Om)} \right) \|\wt u\|_{L^1(0,T;\HH)} \nonumber \\
& \le &  C \e^{CT} \left( \norm{a-\wt a}_{C^1(\ov\Om)} + \norm{q-\wt q}_{C^0(\ov\Om)} \right), \label{z1}
\end{eqnarray}
with the aid of \eqref{eq-est-u1}.
On the other hand, for all $t \in (0,T)$ we have
\begin{eqnarray*}
\norm{\D_t^{(\wt\mu - \mu)}\wt u(t)}_{\LL} & \le &  \norm{\wt\mu-\mu}_{L^\infty(0,1)} \int_0^1 \norm{\pa_t^\al \wt u(t)}_{\LL} \da \\
& \le & \norm{\wt\mu-\mu}_{L^\infty(0,1)}  \int_0^1 
\f{1}{\Ga(1-\al)} \left( \int_0^t \f{\norm{\pa_\tau\wt u(\tau)}_{\LL}}{(t-\tau)^\al}\d\tau \right) \da,
\end{eqnarray*}
from the very definition of $\D_t^{(\mu)}$, with 
$\norm{\int_0^t \f{\norm{\pa_\tau \wt u(\tau)}_{\LL}}{(t-\tau)^\al} \d\tau}_{L^1(0,T)} \le \norm{\pa_t \wt u}_{L^1(0,T;\LL)} \int_0^T {t^{-\al}} \d t$, hence
\begin{eqnarray}
\norm{\D_t^{(\wt\mu - \mu)}\wt u}_{L^1(0,T;\LL)}
& \le & \norm{\wt\mu-\mu}_{L^\infty(0,1)} \norm{\pa_t \wt u}_{L^1(0,T;\LL)} 
\int_0^1 \f{T^{1-\al}}{(1-\al)\Ga(1-\al)} \da \nonumber \\
& \le & C \e^{C T} \norm{\wt\mu-\mu}_{L^\infty(0,1)}  T^{\al_0\ga}
\int_0^1 \f{T^{1-\al}}{\Ga(2-\al)}\da, \label{z2}
\end{eqnarray}
by \eqref{eq-est-u1} and the identity
$\Ga(2-\al) = (1-\al)\Ga(1-\al)$.
Since the mapping $\alpha \mapsto \Ga(\alpha)$ is lower bounded by a positive constant, uniformly in the interval $[1,2]$, we deduce from \eqref{z2} that
$$
\norm{\D_t^{(\wt\mu - \mu)}\wt u}_{L^1(0,T;\LL)}
\le C \e^{C T} \f{T^{1+\al_0\ga}}{|\log T|} \|\wt\mu-\mu\|_{L^\infty(0,1)}.
$$
Putting this together with \eqref{z0} and \eqref{z1}, we obtain that
$$
\|F\|_{L^1(0,T;\LL)} 
\le C \e^{CT} \left( \f{T^{1+ \al_0\ga}}{|\log T|} \|\mu-\wt\mu\|_{L^\infty(0,1)}
+T^\ga \left( \norm{a-\wt a}_{C^1(\ov\Om)} + \norm{q-\wt q}_{C^0(\ov\Om)} \right) \right).
$$
With reference to \eqref{esti_u-h2}, this entails for all $\ka\in[0,1)$ and all $p \in \left[1,\f{1}{1-\al_0(1-\ka)}\right)$, that
$$
\|u-\wt u\|_{L^p(0,T;H^{2 \ka}(\Om))}
\le C \e^{CT} \left( \f{T^{1+\al_0\ga}}{|\log T|}\|\mu-\wt\mu\|_{L^\infty(0,1)}
+T^\ga \left( \norm{a-\wt a}_{C^1(\ov\Om)} + \norm{q-\wt q}_{C^0(\ov\Om)} \right) \right),
$$
which yields the desired result.

\section{Appendix}
\label{sec-app}

\subsection{Proof of Lemma \ref{lem-G_n}}
We stick with the notations used in the proof of Lemma \ref{lem-strichartz-nonhomo} and recall from \eqref{c1} that
$$ 
G_{n,+}(t) + G_{n,-}(t) 
= -\f{1}{\pi} 
\int_\ve^{+\infty} \Im \left( 
\f{\e^{\i\te+rt\e^{\i\te}}}{r\e^{\i\te}w(r\e^{\i\te})+\la_n} \right) \dr,\ \ve \in (0,1).
$$
Since the integrand in the above integral reads
$$
\Im \left( \e^{\i\te + rt \e^{\i\te}} \right)
\Re\left(\f{1}{r\e^{\i\te}w(r\e^{\i\te})+\la_n} \right)
+\Re \left( \e^{\i\te + rt\e^{\i\te}} \right)
\Im\left(\f{1}{r\e^{\i\te}w(r\e^{\i\te})+\la_n} \right)
$$
and that
$
\lim_{\te \to \pi} \Im \left( \e^{\i\te+rt\e^{\i\te}}\right)
= \lim_{\te \to \pi} \e^{rt\cos\te} \sin(\te + rt\sin\te ) = 0,
$
we derive from \eqref{c1} that
\begin{equation}
\label{g1}
G_n(t)=
\f{1}{2 \pi \i} \lim_{\te \to \pi} \int_{\ga_c(\ve,\te)} \f{1}{sw(s) + \la_n} \e^{st}\ds
+\f{1}{\pi}\int_\ve^{+\infty} \Phi_n(r) \e^{-rt} \dr,
\end{equation}
where $\Phi_n(r) :=\Im\left(\f{1}{r\e^{\i\pi}w(r\e^{\i\pi})+\la_n}\right)$ is expressed by \eqref{g0b}.
Here, we used the fact that $G_n(t)=\f{1}{2 \pi \i}\lim_{\te \to \pi} \int_{\ga(\ve,\te)} \f{1}{s w(s) + \la_n} \e^{s t} \d s$, as the integral in the right hand side of \eqref{c0} is independent of $\te \in ( \f{\pi}{2},\pi)$. 

Further, since $\int_{\ga_c(\ve,\te)} \f{\e^{st}}{sw(s) + \la_n} \ds= \int_{-\te}^{\te} f_n(\ve,\be) \d \be$ with $f_n(\ve,\be):=\f{i \ve \e^{i \be} \e^{\ve t \e^{i \be}}}{\int_0^1 \ve^\al \e^{i \al \be} \mu(\al) \d \al + \la_n}$, and $\abs{f_n(\ve,\be)} \le \f{2 \ve \e^{\ve t}}{\la_n}$ for all $\be \in [-\pi,\pi]$ provided $\ve \in \left( 0, \min \left(1, (\eta \la_1)^{\f{1}{\al}} \right) \right)$, then we get
$$
\lim_{\ve \to 0} \left( \lim_{\te \to \pi}  \int_{\ga_c(\ve,\te)} \f{\e^{st}}{sw(s) + \la_n} \ds \right)=0,
$$
by applying the dominated convergence theorem.
Thus, bearing in mind that the right hand side of \eqref{g1} is independent of $\ve \in ( 0, +\infty)$ since this is the case for the one of \eqref{c0}, we get \eqref{g0a}-\eqref{g0b} by sending $\ve$ to $0$ in \eqref{g1}.
 
We turn now to proving \eqref{g0c}. 
To this purpose, for all $n \in \N$, we introduce $a_n \in (0,+\infty)$ such that
\begin{equation}
\label{g2} 
\int_0^1 a_n^\al \mu(\al) \da = \f{\la_n}{2}.
\end{equation}
Notice that the positive real number $a_n$ is well defined for every $n \in \N$, as the mapping $h : r \mapsto \int_0^1 r^\al \mu(\al) \da$ is one-to-one from $[0,+\infty)$ onto itself. Moreover, we point out for further use that 
\begin{equation}
\label{g2b}
\lim_{n \to +\infty} a_n = +\infty.
\end{equation}
This can be understood from the facts that $a_n =h^{-1} \left( \f{\la_n}{2} \right)$, where $h^{-1}$ denotes the function inverse to $h$, that the mapping $r \mapsto h(r)$ is increasing on $[0,+\infty)$, and that $\lim_{r \to +\infty} h(r)=\lim_{n \to +\infty} \la_n=+\infty$. 
Next, using \eqref{g2}, we get for all $r \in [0,a_n]$ that
$\abs{\int_0^1 r^\al \cos ( \pi \al ) \mu(\al) \da}
\le \abs{\int_0^1 r^\al \mu(\al) \da}
\le \int_0^1 a_n^\al \mu(\al) \da
\le \f{\la_n}2$,
whence
$\int_0^1 r^\al \cos ( \pi \al ) \mu(\al) \da + \la_n \geq \f{\la_n}2$.
Putting this together with \eqref{g0b} and the estimate $\sin u \le u$, which holds true for all $u \in [0, \pi]$, we find for every $r \in (0,a_n]$ that
$$
\f{\Phi_n(r)}{r} 
\le \f{4}{\la_n^2} \int_0^1 r^{\al-1} \sin ( \pi \al ) \mu(\al) \da 
\le \f{4 \pi }{\la_n^2} \int_0^1 \al r^{\al-1} \mu(\al) \da 
\le 
\f{4\pi}{\la_n^2} \f{\d}{\d r} \left( \int_0^1 r^{\al} \mu(\al) \da \right). 
$$
From this and \eqref{g2}, it then follows that
\begin{equation}
\label{g3}
\int_0^{a_n} \f{\Phi_n(r)}{r} \d r 
\le \f{4 \pi}{\la_n^2} \int_0^1 a_n^\al \mu(\al) \da 
\le \f{2\pi}{\la_n}. 
\end{equation}

The next step of the proof boils down to the fact that there exist two positive constants $C_0$ and $R_0$, both of them depending only on $\mu$, such that we have
\begin{equation}
\label{esti-lhospital}
\int_R^{+\infty} \f{1} {\int_0^1 r^{\al+1} \sin(\pi\al)\mu(\al)\da}\dr
\le \f{C_0}{\int_0^1 R^{\al} \sin(\pi\al)\mu(\al)\da},\ R \in [R_0,+\infty).
\end{equation}
The proof of \eqref{esti-lhospital}, which is quite similar to the derivation of L'Hospital's rule, is presented in Section \ref{sec-app2} for the convenience of the reader. Prior to applying \eqref{esti-lhospital} and upon possibly enlarging $R_0$, we notice from \eqref{cnd-mu1} that we have
\begin{equation}
\label{esti-mu}
\int_0^{\al_0-\de} r^\al\mu(\al)\da 
\le \int_{\al_0-\de}^1 r^\al\mu(\al)\da,\ r \in [R_0,+\infty).
\end{equation}
This follows directly from the two basic estimates
$\int_0^{\al_0-\de} r^\al\mu(\al)\da \le \norm{\mu}_{L^\infty(0,1)} (\al_0-\de) r^{\al_0-\de}$
and
$\int_{\al_0-\de}^1 r^\al\mu(\al)\da \ge \int_{\al_0-\f{\de}{2}}^{\al_0-\de} r^\al\mu(\al)\da \ge \f{\de \mu(\al_0)}{4} r^{\al_0-\f{\de}{2}}$, which are valid for all $r \in (0,+\infty)$.

Now, in light of \eqref{g2b}, we pick  $N \in \N$ so large that $a_N \ge R_0$ and we apply \eqref{esti-lhospital} with $R=a_n$ for all $n \in \N_N := \{ n \in \N,\ n \ge N \}$. With reference to \eqref{g0b}, we obtain that
$$
\int_{a_n}^{+\infty} \f{\Phi_n(r)}{r} \dr
\le \int_{a_n}^{+\infty} \f{1} {\int_0^1 r^{\al+1} \sin(\pi\al)\mu(\al)\da}\dr
\le \f{C_0}{\int_0^1 a_n^{\al}\sin(\pi\al) \mu(\al)\da},\ n \in \N_N.
$$
In view of \eqref{cnd-mu2}, this leads to
\begin{equation}
\label{g5}
\int_{a_n}^{+\infty} \f{\Phi_n(r)}{r} \dr
\le \f{C_0}{\int_{\al_0-\de}^{\al_1} a_n^{\al}\sin(\pi\al) \mu(\al)\da}
\le \f{C}{\int_{\al_0-\de}^{\al_1} a_n^{\al} \mu(\al)\da}
\le \f{C}{\int_{\al_0-\de}^1 a_n^{\al} \mu(\al)\da},\ n \in \N_N,
\end{equation}
by setting $C:= \f{C_0}{\min \left( \sin( \pi (\al_0-\de) ) , \sin ( \pi \al_1 ) \right)} \in (0,+\infty)$. Next, applying \eqref{esti-mu} with $r=a_n$ and $n \in \N_N$, which is permitted since $a_n \ge a_N \ge R_0$ for all $n \in \N_N$, we get that $\int_{\al_0-\de}^1 a_n^{\al} \mu(\al)\da \geq \f{\int_0^1 a_n^{\al} \mu(\al)\da}{2}$. Therefore, we have
$\int_{a_n}^{+\infty} \f{\Phi_n(r)}{r} \dr \le \f{C}{\la_n}$ for all $n \in \N_N$, from \eqref{g2} and \eqref{g5}. This entails \eqref{g0c} since 
$N$ is finite.

\subsection{Proof of the estimate \eqref{esti-lhospital}}
\label{sec-app2}
For $R \in (0,+\infty)$, put
$f(R):=\int_R^{+\infty}\f{1} {\int_0^1 r^{\al+1} \sin(\pi\al)\mu(\al)\da} \dr$ and
$g(R):=\f{1}{\int_0^1 R^{\al} \sin(\pi\al)\mu(\al)\da}$.
Evidently, $f$ and $g$ are two positive functions in $(0,+\infty)$ that vanish at infinity: 
\begin{equation}
\label{eq-lim}
\lim_{R\to+\infty}f(R)=\lim_{R\to+\infty}f(R)=0,
\end{equation}
and for all $R \in (0,+\infty)$ and all $R_1 \in(R,+\infty)$, we find by applying Rolle's theorem to the function $r \to (f(r)-f(R_1))(g(R)-g(R_1))-(f(R)-f(R_1))(g(r)-g(R_1))$ on the interval $[R,R_1]$, that there exists $\xi\in (R,R_1)$ such that we have
\begin{equation}
\label{lhr}
\f{f(R) - f(R_1)}{g(R) - g(R_1)} 
=\f{f'(\xi)}{g'(\xi)}.
\end{equation}
Next, since $\int_0^{\al_0-\de} r^\al \sin(\pi\al) \mu(\al) \d \al \le (\al_0-\de) \| \mu \|_{L^\infty(0,1)} r^{\al-0-\de}$ and
$$\int_{\al_0-\de}^1 r^\al \sin(\pi\al) \mu(\al) \d \al \ge 
\int_{\al_0-\f{\de}{2}}^{\al_0} r^\al \sin(\pi\al) \mu(\al) \d \al \ge c r^{\al_0-\f{\de}{2}},$$ 
by \eqref{cnd-mu1}, where $c:= \min \left( \sin \left( \left(\al_0-\f{\de}{2} \right) \pi \right), \sin(\al_0 \pi) \right)  \f{\de \mu(\al_0)}{4} \in (0,+\infty)$, then there is necessarily $R_0 \in [1,+\infty)$, such that
$$
\int_0^{\al_0-\de} r^\al \sin(\pi \al) \mu(\al) \da 
\le \int_{\al_0-\de}^1 r^\al \sin(\pi \al) \mu(\al) \da,\ r \in [R_0,+\infty).
$$
From this and the identity
$\f{f'(r)}{g'(r)} = \f{\int_0^1 r^\al \sin(\pi\al)\mu(\al)\da}{\int_0^1 \al r^\al \sin(\pi\al)\mu(\al)\da}$, obtained for all $r \in (0,+\infty)$ by direct calculation, it then follows that
$$
\f{f'(r)}{g'(r)}\le\f{2\int_{\al_0-\de}^1 r^\al \sin(\pi\al)\mu(\al)\da}{(\al_0-\de)\int_{\al_0-\de}^1 r^\al \sin(\pi\al)\mu(\al)\da}\le \f{2}{\al_0-\de},\ r \in [R_0,+\infty).
$$
Putting this with \eqref{lhr} we get
$$
\f{f(R) - f(R_1)}{g(R) - g(R_1)} 
=\f{f'(\xi)}{g'(\xi)}\le \f{2}{\al_0-\de},\ R \in [R_0,+\infty),
$$
uniformly in $R_1 \in (R,+\infty)$. Thus, taking $R_1\to +\infty$ in the above estimate, we obtain
$$
\lim_{R_1\to+\infty} \f{f(R) - f(R_1)}{g(R) - g(R_1)} 
=\f{f(R)}{g(R)} \le \f{2}{\al_0-\de},\ R \in [R_0,+\infty),
$$
in virtue of \eqref{eq-lim}, which is the statement of \eqref{esti-lhospital}. 

\section*{Acknowledgment}
This work is partially supported by Grant-in-Aid for Scientific Research (S) 15H05740, JSPS.
The first author thanks Grant-in-Aid for Research Activity Start-up 16H06712, JSPS.
The two last authors would like to thank the Department of Mathematical Sciences of The University of Tokyo, where part of this article was written, for its kind hospitality.

\end{document}